\documentclass[11pt, leqno]{amsart}
\usepackage[utf8]{inputenc}
\usepackage[T1]{fontenc}

\counterwithout{subsection}{section}

\theoremstyle{plain}
\newtheorem{theorem}{Theorem}
\newtheorem{proposition}[theorem]{Proposition}
\newtheorem{lemma}[theorem]{Lemma}
\newtheorem{corollary}[theorem]{Corollary}

\theoremstyle{definition}
\newtheorem{definition}[theorem]{Definition}
\newtheorem{convention}[theorem]{Convention}

\theoremstyle{remark}
\newtheorem{remark}[theorem]{Remark}
\newtheorem{example}[theorem]{Example}

\usepackage{xpatch}
\NewCommandCopy{\notocsection}{\section}
\xpatchcmd{\notocsection}{{1}}{{1001}}{}{}

\makeatletter 
\def\l@subsection{\@tocline{2}{0pt}{2pc}{6pc}{}}
\makeatother

\makeatletter
\newif\ifshowtocpage
\showtocpagetrue

\let\oldtocline\@tocline
\def\@tocline#1#2#3#4#5#6#7{%
    \ifshowtocpage \oldtocline{#1}{#2}{#3}{#4}{#5}{#6}{#7}%
    \else \oldtocline{#1}{#2}{#3}{#4}{#5}{#6}{}%
    \fi%
}
\makeatother

\newcommand{\hidepagenumber}{\addtocontents{toc}{\protect\showtocpagefalse}}
\newcommand{\showpagenumber}{\addtocontents{toc}{\protect\showtocpagetrue}}

\usepackage{tikz}
\usetikzlibrary{patterns, fit}
\usepackage{tikz-cd}
\usepackage{quiver}
\usepackage{amsmath}
\usepackage{amsthm}
\usepackage{amssymb}
\usepackage{stmaryrd}
\usepackage{mathtools}
\usepackage{bbold}
\usepackage{systeme,mathtools}
\usepackage[hidelinks]{hyperref}
\usepackage{enumitem}

\usepackage[
    backend=biber,
    style=alphabetic,
    maxnames=10,
    maxalphanames=10,
    sorting=nyt,
    url=false, doi=false, isbn=false%
]{biblatex}
\AtBeginBibliography{\small}

\DeclareFieldFormat{extraalpha}{#1}

\DeclareLabelalphaTemplate{
  \labelelement{
    \field[final]{shorthand}
    \field{label}
    \field[strwidth=3,strside=left,ifnames=1]{labelname}
    \field[strwidth=1,strside=left]{labelname}
  }
}

\DeclareSymbolFont{sfoperators}{OT1}{cmss}{m}{n}
\DeclareSymbolFontAlphabet{\mathsf}{sfoperators} \makeatletter
\def\operator@font{\mathgroup\symsfoperators} \makeatother

\newcommand{\Z}{\mathbb Z}
\DeclareMathOperator{\End}{End}
\DeclareMathOperator{\Hom}{Hom}
\DeclareMathOperator{\im}{im}

\DeclareMathOperator{\rad}{rad}
\DeclareMathOperator{\soc}{soc}
\DeclareMathOperator{\head}{top}
\DeclareMathOperator{\modcat}{mod}
\DeclareMathOperator{\projcat}{proj}
\DeclareMathOperator{\injcat}{inj}
\DeclareMathOperator{\Ext}{Ext}
\DeclareMathOperator{\pdim}{proj\,dim}
\DeclareMathOperator{\idim}{inj\,dim}

\DeclareMathOperator{\add}{add}
\DeclareMathOperator{\silt}{silt}
\DeclareMathOperator{\twosilt}{2-silt}
\DeclareMathOperator{\Fac}{Fac}
\DeclareMathOperator{\Sub}{Sub}
\DeclareMathOperator{\tors}{tors}
\DeclareMathOperator{\torf}{torf}
\DeclareMathOperator{\ftors}{ff-tors}
\DeclareMathOperator{\ftorf}{ff-torf}
\DeclareMathOperator{\ind}{ind}
\DeclareMathOperator{\stors}{s-tors}
\DeclareMathOperator{\red}{red}
\DeclareMathOperator{\lift}{lift}

\DeclareMathOperator{\Fun}{Fun}
\DeclareMathOperator{\Cam}{C}
\DeclareMathOperator{\thick}{thick}
\newcommand{\inv}{^{-1}}
\newcommand{\opcat}{^\mathrm{op}}

\newcommand{\surj}{\twoheadrightarrow}
\newcommand{\inj}{\hookrightarrow}
\newcommand{\iso}{\stackrel{\sim}{\smash{\longrightarrow}\rule{0pt}{0.4ex}}}
\newcommand{\card}[1]{\lvert #1 \rvert}
\newcommand{\blank}{{-}}
\newcommand{\deriv}{\mathcal D^{\mathrm b}}
\newcommand{\Kb}{\mathcal K^{\mathrm b}}
\newcommand{\Kbplus}{\mathcal K^{\mathrm b, +}}
\newcommand{\leftperp}[1]{\prescript{\perp}{}{#1}}
\newcommand{\rightperp}[1]{#1^\perp}
\newcommand{\twosiltwrt}[1]{2_#1\operatorname{-silt}}
\newcommand{\longrightleftarrows}{\mathrel{\substack{\longrightarrow \\[-.6ex] \longleftarrow}}}
\newcommand{\plus}{+}

\begin{document}

\title{Derived equivalence of posets of torsion classes}
\author{Marius Goguet}
\address{Universit\'e Grenoble Alpes, Institut Fourier, CS 40700, 38058 Grenoble cedex 09}
\email{marius.goguet@univ-grenoble-alpes.fr}

\begin{abstract}
    Inspired by work of Ladkani,
    we investigate the structure of the poset of (functorially finite) torsion classes for a finite dimensional algebra admitting a simple projective module.
    We extend Ladkani's result by showing that if two algebras are related by a 1-APR tilt,
    then their posets of functorially finite torsion classes are related by a flip-flop.
    This implies that the incidence algebras of the posets are derived equivalent.
    We also show that the result holds for the posets of arbitrary torsion classes.
    A key ingredient of the proof is to see the two posets we want to relate
    as subposets of a common poset of silting objects in the functorially finite case,
    and of a common poset of s-torsion pairs in the arbitrary case.
\end{abstract}

\begin{abstract}
    We investigate the structure of the poset of torsion classes for a finite dimensional algebra admitting a simple projective module.
    We generalise a result of Ladkani by showing that if two algebras are related by a 1-APR tilt,
    then their posets of torsion classes are related by a flip-flop.
    This implies that the incidence algebras of the posets are derived equivalent.
    We give two different proofs of this result.
    The first one applies to functorially finite torsion classes,
    and a key ingredient is to see the two posets we want to relate as subposets of a common poset of silting objects.
    The second proof applies to arbitrary torsion classes,
    and we use a similar strategy,
    this time embedding the two posets into a common poset of s-torsion pairs.
\end{abstract}

\maketitle

\tableofcontents

\section*{Introduction}

The notion of torsion pairs in abelian categories was introduced by Dickson in \cite{Dickson},
as a generalisation of the theory of torsion(-free) abelian groups.
The study of torsion pairs in the module category of an algebra, hand in hand with tilting theory,
has been at the centre of the development of the modern representation theory of finite dimensional algebras.
For instance, torsion pairs feature prominently in the Brenner-Butler theorem \cite{BB},
and they were used by Happel, Reiten and Smal\o{} in \cite{HRS} to construct new t-structures
in derived categories of algebras.

The set of torsion classes for a finite dimensional algebra $\Lambda$ comes with a natural partial ordering,
and the resulting poset can be seen as a combinatorial invariant of $\Lambda$.
This poset enjoys many nice and interesting properties,
of which we mention two.
First, it is a complete lattice, and the relationship between its lattice structure and the representation theory of $\Lambda$
has been recently studied in depth, see for instance \cite{DIRRT}.
Second, under suitable conditions,
the set of torsion classes is in bijection with several other sets of objects associated to the representation theory of $\Lambda$,
such as support $\tau$-tilting modules, 2-term silting complexes of projective modules, cluster-tilting objects in the cluster category,
to name a few.
This fact was at the centre of the introduction of $\tau$-tilting theory in \cite{AIR}.
Moreover, for each of these sets, there is a notion of partial order,
which turns out to be respected by those bijections.

To a finite poset, one can associate a finite dimensional algebra,
called the incidence algebra of the poset.
The study of posets through their incidence algebras is an area of active research:
by studying how homological properties of the incidence algebra relate to combinatorial properties of the poset
\cite{IM, GKKM};
studying properties of the derived category of the incidence algebra, such as the fractionally Calabi-Yau property
\cite{Cha,Rog,Got};
or studying when two incidence algebras share the same derived category \cite{Lad-sheaves}.

In \cite{Lad-universal}, Ladkani introduced a combinatorial operation on posets,
called a flip-flop.
Roughly speaking, two posets $Z$ and $Z'$ are said to be related by a flip-flop
when $Z$ can be split into two disjoint subposets $X$ and $Y$,
so that $Z'$ is obtained by reversing the relative order of $X$ and $Y$ in the poset,
but keeping the same ordering inside of $X$ and $Y$ (see Figure~\ref{figure intro}).
Moreover, $Z$ and $Z'$ can be reconstructed from the data of the posets $X$ and $Y$,
and of a certain order-preserving map $f : X \to Y$.
Ladkani proved that if two finite posets are related by a flip-flop,
then the derived categories of their incidence algebras are equivalent.
Moreover, under a small extra assumption,
one can show that the Hasse quivers of the posets can be easily deduced from each other.
Ladkani's construction was later generalised in \cite{DJW} to the context of stable $\infty$-categories.

\begin{figure}
    \def\svgwidth{0.92\textwidth}
    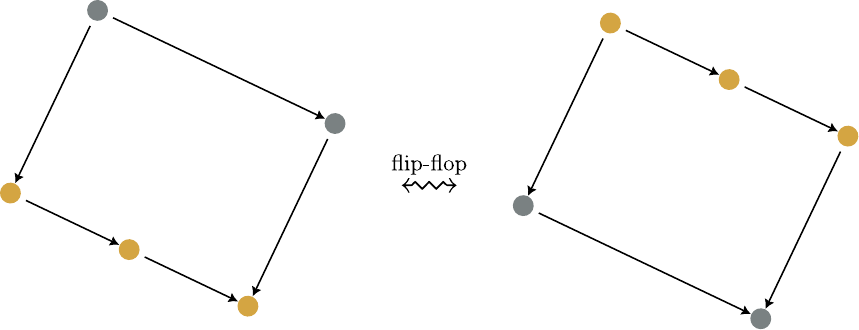
    \caption{Two posets related by a flip-flop.}
    \label{figure intro}
\end{figure}

In \cite{Lad-cluster_tilting} and \cite{Lad-tilting},
Ladkani applied this framework to some posets appearing naturally in representation theory.
In particular, he showed that for two quivers $Q$ and $Q'$ related by a reflection,
the posets of functorially finite torsion classes $\ftors kQ$ and $\ftors kQ'$ (under a different name)
are related by a flip-flop.

The goal of this paper is to extend this result in two different directions.
First, we broaden the class of pairs of algebras whose posets of torsion classes we are able to compare,
from hereditary algebras to algebras of arbitrary global dimension.
To this end, we replace the notion of reflection of quivers by the more general notion of APR tilting.
This was introduced by Auslander, Platzeck and Reiten in \cite{APR},
with the goal of giving a more categorical approach to Bernstein--Gelfand--Ponomarev reflection functors.
If $\Lambda$ is a finite dimensional algebra,
an APR tilting module is a particular tilting $\Lambda$-module associated to a simple projective $\Lambda$-module.
In order to be able to compare the posets of torsion classes of two algebras related by an APR tilt,
we have to impose an extra homological assumption on the simple projective module,
namely that its injective dimension is equal to 1.

In \cite{IO},
Iyama and Oppermann introduced a generalisation of APR tilting modules in the setting of higher homological algebra,
which they called $n$-APR tilting modules, where $n \geq 1$ is an integer.
In the case $n = 1$, what they call a weak 1-APR tilting module is simply a classical APR tilting module,
and they call such a module a 1-APR tilting module under the extra assumption
that the injective dimension of the associated simple projective module is equal to 1.
Since this is exactly the assumption we need, we employ their terminology in the rest of this paper.

In this setting, we show the following.

\begin{theorem} \label{thm intro ff}
    Let $\Lambda$ and $\Gamma$ be two finite dimensional algebras
    which are related by a 1-APR tilt.
    Then their posets of functorially finite torsion classes $\ftors \Lambda$ and $\ftors \Gamma$ are related by a flip-flop.
\end{theorem}

The second direction in which we extend Ladkani's result is by also taking into account the poset of arbitrary torsion classes.
Namely, we show that this flip-flop relationship also exists between these \emph{a priori} larger posets.

\begin{theorem} \label{thm intro arbitrary}
    Let $\Lambda$ and $\Gamma$ be two finite dimensional algebras
    which are related by a 1-APR tilt.
    Then their posets of torsion classes $\tors \Lambda$ and $\tors \Gamma$ are related by a flip-flop.
\end{theorem}

Although the flip-flop decompositions for both types of posets are related,
we stress that Theorem~\ref{thm intro arbitrary} does not imply Theorem~\ref{thm intro ff}.
We also remark that these results do not hold in general if we replace ``1-APR tilt'' by ``APR tilt''.
Indeed, already for $\Gamma$ the path algebra of the equioriented $A_3$ quiver and $\Lambda = \Gamma / (\rad \Gamma)^2$,
one can check that the posets $\tors \Lambda$ and $\tors \Gamma$ have 12 and 14 elements respectively,
which prevents them from being related by a flip-flop,
despite the algebras $\Lambda$ and $\Gamma$ being related by an APR tilt.

We now give an idea of the strategy of the proof.
The first step is to define the flip-flop decompositions for the different posets at play.
These decompositions actually hold in a more general setting than that of 1-APR tilting.
Indeed, they only require that the algebra $\Lambda$ admit a simple projective module,
or dually a simple injective module.

\begin{theorem} \label{thm intro flip-flop decomposition}
    Let $\Lambda$ be a finite dimensional algebra which admits a simple projective module.
    Then the poset $\tors \Lambda$ admits a flip-flop decomposition,
    which restricts to a flip-flop decomposition of $\ftors \Lambda$.
\end{theorem}

While a given poset admits \emph{a priori} many different flip-flop decompositions,
the one given by Theorem~\ref{thm intro flip-flop decomposition} is relevant to us because it comes from representation theory.
Indeed, we decompose the poset of torsion classes into two subposets along a representation-theoretic criterion,
namely whether the simple projective module belongs to a given torsion class or not.
Moreover, we remark that one of those two subposets can be further understood as the poset of torsion classes for a smaller algebra,
through the process of reduction described by Jasso in \cite{Jasso}
(see Remark~\ref{Jasso reduction} for a precise statement).
Thus we can hope that this flip-flop decomposition helps understand more about the structure of the poset of torsion classes,
even when we do not know whether the flip-flop counterpart poset has a representation-theoretic interpretation.

The common strategy for our proofs of Theorem~\ref{thm intro ff} and Theorem~\ref{thm intro arbitrary}
is to leverage the fact that two algebras $\Lambda$ and $\Gamma$ related by a 1-APR tilt are derived equivalent.
In order to fully profit from this fact, we translate our posets of torsion classes $\ftors \Lambda$, $\ftors \Gamma$, and $\tors \Lambda$, $\tors \Gamma$
into posets of objects or subcategories that live in the ``common'' derived category $\mathcal D$ of $\Lambda$ and $\Gamma$.

In the functorially finite case, we see $\ftors \Lambda$ and $\ftors \Gamma$ as posets of 2-term silting complexes in the homotopy category $\Kb (\projcat \Lambda)$,
in the spirit of $\tau$-tilting theory, following \cite{AIR}.
In this context, we are able to interpret the order-preserving maps defining the flip-flop decompositions as silting mutation.
This allows us to use the rich theory of silting complexes and the nice properties of completion exhibited by $\tau$-tilting theory.

In the case of arbitrary torsion classes,
we use the framework of extriangulated categories,
which was introduced by Nakaoka and Palu in \cite{NP}.
We introduce an extriangulated category $\mathcal C$ as a certain extension-closed subcategory of the derived category $\mathcal D$.
We think of this extriangulated category as being obtained by ``gluing'' the module categories $\modcat \Lambda$ and $\modcat \Gamma$ along a certain functor $\modcat \Lambda \to \modcat \Gamma$.
In \cite{AET}, Adachi, Enomoto and Tsukamoto introduced the notion of s-torsion pairs in an extriangulated category $\mathcal E$
equipped with an extra structure of first negative extension.
They developed a theory of reduction for intervals in the poset $\stors \mathcal E$ of s-torsion pairs,
giving a framework that encompasses Jasso's reduction of torsion pairs,
as well as the bijection between intermediate t-structures and torsion pairs discovered by Happel, Reiten and Smal\o{}.
In our context, the poset $\stors \mathcal C$ appears as a particularly convenient framework to compare the posets $\tors \Lambda$ and $\tors \Gamma$.
Indeed, the poset $\stors \mathcal C$ can also be seen as a kind of ``gluing'' of the two posets of torsion classes.
Moreover, we make heavy use of the reduction theorem of \cite{AET} to understand the structure of $\stors \mathcal C$
and to relate the different parts of the flip-flop decompositions at play.

To conclude, we give an application of our results to a class of well-studied lattices that arise as posets of torsion classes,
called Cambrian lattices.
They were introduced by Reading in \cite{Rea},
generalising the Tamari lattice (in equioriented type $A$) to any Dynkin type and any orientation.
In \cite{Lad-cluster_tilting}, Ladkani applied his results to show that the Cambrian lattices $\Cam_{\Delta,\Omega}$
corresponding to various orientations $\Omega$ of a fixed simply-laced Dynkin diagram $\Delta$ are all derived equivalent.
Our results allow us to extend this application to non simply-laced Dynkin diagrams,
by considering 1-APR tilts of representation-finite hereditary algebras given by species,
or alternatively 1-APR tilts of the (non-hereditary) algebras introduced by Gei{\ss}, Leclerc and Schr{\"o}er.

We give an overview of the structure of the paper.
In sections 1 through 5 we give necessary preliminaries on torsion pairs, silting objects, 1-APR tilting and flip-flop.
In section 6, we introduce the flip-flop decompositions of the posets $\tors \Lambda$ and $\ftors \Lambda$ in the presence of a simple projective module,
and prove Theorem~\ref{thm intro flip-flop decomposition}.
In section 7 we translate this result in terms of 2-term silting complexes of projective modules,
and in section 8 we use this translation to prove Theorem~\ref{thm intro ff}.
In section 9, we explain how to apply our results to get derived equivalences of Cambrian lattices.
In section 10, we give the necessary preliminaries on extriangulated categories, s-torsion pairs, and t-structures in triangulated categories.
In section 11, we introduce an extriangulated category $\mathcal C$ associated to the data of a 1-APR tilt.
Finally, in section 12, we study the structure of the poset $\stors \mathcal C$,
and we use it to prove Theorem~\ref{thm intro arbitrary}.

\notocsection{Conventions}

Throughout the paper, we fix a field $k$.
All algebras we consider are basic and finite dimensional over $k$,
and for such an algebra $\Lambda$, we write $\modcat \Lambda$ for the category of finitely generated right $\Lambda$-modules.
We compose arrows in a quiver in the same direction as composition of maps,
that is, $ba$ denotes the path starting at the source of $a$ and ending at the target of $b$.
We denote by $D = \Hom_k (\blank, k)$ the $k$-linear duality functor.

All categories we consider are additive, $k$-linear, Hom-finite,
and satisfy the Krull--Remak--Schmidt property.
If $\mathcal C$ is such a category,
by \emph{subcategory} of $\mathcal C$ we always mean a full subcategory which is closed under isomorphisms,
taking direct sums and direct summands.
With this convention, a subcategory $\mathcal X$ of $\mathcal C$ is determined
by its set of (isomorphism classes of) indecomposable objects,
which we denote by $\ind \mathcal X$,
and which coincides with the set $(\ind \mathcal C) \cap \mathcal X$.
If $E$ is a set of objects in $\mathcal C$,
we denote by $\add E$ the smallest subcategory of $\mathcal C$ containing the elements of $E$.
If $\mathcal X$ and $\mathcal Y$ are two subcategories of $\mathcal C$,
we write $\mathcal X \plus \mathcal Y = \add (\mathcal X \cup \mathcal Y)$,
which is characterised as a subcategory of $\mathcal C$ by $\ind (\mathcal X \plus \mathcal Y) = (\ind \mathcal X) \cup (\ind \mathcal Y)$.

If $(P,\leq)$ is a partially ordered set (poset),
its Hasse quiver is the quiver whose set of vertices is $P$,
and whose arrows are of the form $x \to y$ whenever $x > y$ is a cover relation in $P$,
meaning there is no $z \in P$ satisfying $x > z > y$.

\notocsection{Acknowledgements}

I would like to thank my advisors, Claire Amiot and Baptiste Rognerud,
for introducing me to this problem, and for their constant guidance and support throughout this project.

\numberwithin{theorem}{subsection}

\hidepagenumber
\section{Preliminaries} \label{sec:preliminaries}
\showpagenumber

\subsection{Torsion classes and torsion-free classes}

In this section, we define the notions of torsion pair, torsion class, torsion-free class
in the module category of a finite dimensional algebra,
and we also define associated posets.

The notion of torsion pair in an abelian category was introduced by Dickson in \cite{Dickson},
as a generalisation of the notion of torsion and torsion-free abelian groups.

\begin{definition} \label{def torsion pair}
    A pair $(\mathcal T, \mathcal F)$ of subcategories of $\modcat \Lambda$ is called a \emph{torsion pair} when
    \begin{enumerate}
        \item $\Hom_\Lambda(\mathcal T, \mathcal F) = 0$
        \item For every $X \in \modcat \Lambda$, there exists a short exact sequence
        $ 0 \to tX \to X \to fX \to 0$ with $tX \in \mathcal T$ and $fX \in \mathcal F$.
    \end{enumerate}
    If $(\mathcal T, \mathcal F)$ is a torsion pair,
    then $\mathcal T$ is called a \emph{torsion class} and $\mathcal F$ is called a \emph{torsion-free class}.
\end{definition}

The next proposition states that each subcategory in a torsion pair determines the other one.

\begin{proposition}[\protect{\cite{Dickson}}] \label{prop hom orthogonals}
    Let $(\mathcal T, \mathcal F)$ be a torsion pair in $\modcat \Lambda$.
    Then $\mathcal F = \rightperp {\mathcal T}$ and $\mathcal T = \leftperp {\mathcal F}$,
    where we set
    \[\rightperp {\mathcal T} = \{X \in \modcat \Lambda \mid \forall T \in \mathcal T, \Hom (T, X) = 0\}\]
    and
    \[\leftperp {\mathcal F} = \{X \in \modcat \Lambda \mid \forall F \in \mathcal F, \Hom (X, F) = 0\}.\]
\end{proposition}

\begin{corollary}[\protect{\cite{Dickson}}] \label{cor:first results on torsion pairs}
    Let $(\mathcal T, \mathcal F)$ be a torsion pair in $\modcat \Lambda$.
    Then
    \begin{enumerate}
        \item $\mathcal T$ is closed under taking extensions and quotients.
        \item $\mathcal F$ is closed under taking extensions and submodules.
        \item For $(\mathcal T', \mathcal F')$ a torsion pair in $\modcat \Lambda$,
        we have $\mathcal T \subseteq \mathcal T' \iff \mathcal F \supseteq \mathcal F'$.
    \end{enumerate}
\end{corollary}

The next proposition gives a converse of the first two statements in Corollary~\ref{cor:first results on torsion pairs}.
It relies on the fact that $\modcat \Lambda$ is a length category.

\begin{proposition}[\protect{\cite[Proposition~VI.1.4]{ASS}}] \label{prop:characterisation of torsion classes}
    Let $\mathcal T$ be a subcategory of $\modcat \Lambda$,
    and assume that $\mathcal T$ is closed under taking extensions and quotients.
    Then $(\mathcal T, \rightperp{\mathcal T})$ is a torsion pair.
    
    Dually, if $\mathcal F \subseteq \modcat \Lambda$ is closed under taking extensions and submodules,
    then $(\leftperp{\mathcal F}, \mathcal F)$ is a torsion pair.
\end{proposition}

We now define the notion of functorial finiteness for a subcategory of $\modcat \Lambda$,
and give a characterisation of functorially finite torsion(-free) classes,
which is due to Smal\o{}.

\begin{definition}
    Let $\mathcal C$ be a subcategory of $\modcat \Lambda$.
    We say that $\mathcal C$ is \emph{covariantly finite} if every object $M \in \modcat \Lambda$
    admits a left $\mathcal C$-approximation,
    that is a morphism $M \stackrel{f}{\to} C$ with $C \in \mathcal C$
    such that for every $C' \in \mathcal C$,
    the morphism $\Hom(f, C')$ is surjective.
    
    The notion of \emph{contravariantly finite} subcategory is defined dually.
    
    We say that $\mathcal C$ is \emph{functorially finite} when it is both covariantly and contravariantly finite.
\end{definition}

\begin{proposition}[\protect{\cite{Smalo}}] \label{prop:equivalence ff torsion(free)}
    Let $(\mathcal T, \mathcal F)$ be a torsion pair.
    The following are equivalent:
    \begin{enumerate}
        \item $\mathcal T$ is functorially finite.
        \item $\mathcal F$ is functorially finite.
        \item $\mathcal T = \Fac T$ for some $T \in \modcat \Lambda$, 
            where $\Fac T = \{N \in \modcat \Lambda \mid \exists f : T^{\oplus n} \surj N \textnormal{ surjective}\}$.
        \item $\mathcal F = \Sub F$ for some $F \in \modcat \Lambda$,
            where $\Sub F = \{N \in \modcat \Lambda \mid \exists g : N \inj F^{\oplus n} \textnormal{ injective}\}$.
    \end{enumerate}
\end{proposition}

In the following we will be concerned with both functorially finite and arbitrary torsion(-free) classes.
We write $\tors \Lambda$ (resp.\ $\torf \Lambda$) for the set of torsion classes (resp.\ torsion-free classes) of $\modcat \Lambda$,
and we write $\ftors \Lambda$ (resp.\ $\ftorf \Lambda$) for the corresponding sets of functorially finite subcategories.
These sets are naturally equipped with a partial order given by inclusion of subcategories.
In general, not all torsion classes are functorially finite.
In fact, in \cite[Theorem~3.8]{DIJ},
Demonet, Iyama and Jasso showed that the posets $\ftors \Lambda$ and $\tors \Lambda$ coincide if and only if $\ftors \Lambda$ is finite.

\begin{proposition}
    We have an isomorphism of posets $(\tors \Lambda, \subseteq) \simeq (\torf \Lambda, \supseteq)$,
    or equivalently $(\tors \Lambda, \subseteq) \simeq (\torf \Lambda, \subseteq)\opcat$,
    given by $\mathcal T \mapsto \rightperp{\mathcal T}$.

    Moreover, this isomorphism restricts to isomorphisms
    \[(\ftors \Lambda, \subseteq) \simeq (\ftorf \Lambda, \supseteq) = {(\ftorf \Lambda, \subseteq)}\opcat.\]
\end{proposition}

\begin{proof}
    This is a direct consequence of Proposition~\ref{prop hom orthogonals}, Corollary~\ref{cor:first results on torsion pairs} and Proposition~\ref{prop:equivalence ff torsion(free)}.
\end{proof}

\subsection{Silting objects in a triangulated category}

Silting objects (or more generally silting subcategories) in triangulated categories were introduced by Keller and Vossieck in \cite{KV},
and their mutation theory was introduced by Aihara and Iyama in \cite{AI}.

Let $\mathcal D$ be a triangulated category,
whose translation functor is denoted by $\Sigma$.

\begin{definition}
    An object $S \in \mathcal D$ is called \emph{presilting} when
    $\Hom_\mathcal D(S, \Sigma^i S) = 0$ for every $i > 0$.
    A presilting object $S$ which generates $\mathcal D$ as a thick subcategory is called \emph{silting}.
    We denote by $\silt \mathcal D$ the set of (isomorphism classes of) basic silting objects in $\mathcal D$.
\end{definition}

In the following, we will write the vanishing condition ``$\Hom_\mathcal D(S, \Sigma^i S) = 0$ for every $i > 0$''
as ``$\Hom_\mathcal D (S, \Sigma^{>0} S) = 0$'',
and we use similar notation for similar vanishing conditions.

\begin{remark}
    If $S$ is a (pre)silting object in $\mathcal D$,
    then so is $\Sigma^i S$ for every $i \in \Z$.

    Any direct summand of a presilting object is also presilting.
\end{remark}

\begin{definition}
    We define a relation $\geq$ on $\silt \mathcal D$ by setting
    \[S \geq T \iff \Hom_\mathcal D (S, \Sigma^{>0} T) = 0.\]
\end{definition}

\begin{remark}
    If we have $S \geq T$ for some $S, T \in \silt \mathcal D$,
    then clearly we have $\Sigma^i S \geq \Sigma^i T$ for any $i \in \Z$.
\end{remark}

\begin{theorem}[\protect{\cite[Theorem~2.11]{AI}}]
    The relation $\geq$ is a partial order on $\silt \mathcal D$.
\end{theorem}

We now recall the operation called silting mutation.

\begin{definition}
    Let $S = X \oplus S' \in \silt \mathcal D$.
    Take a minimal left ($\add S'$)-approximation $f : X \to T$,
    and form a triangle $X \to T \to Y \to \Sigma X$ in $\mathcal D$.
    Then the \emph{left mutation} of $S$ with respect to $X$ is defined by
    $\mu^-_X(S) = Y \oplus S'$.

    We define \emph{right mutation} dually, that is by taking a minimal right ($\add S'$)-approximation $g : U \to X$,
    forming the triangle $Z \to U \to X \to \Sigma Z$,
    and setting $\mu^+_X(S) = Z \oplus S'$.
\end{definition}

\begin{remark}
    We use the opposite sign convention of \cite{AI} for $\mu^\pm$,
    but we use the same left/right terminology.
    See also the footnote in \cite[Definition-Proposition~1.7]{AIR}.
\end{remark}

\begin{theorem}[\protect{\cite[Theorem~2.31]{AI}}] \label{mutation preserves silting}
    Let $S = X \oplus S' \in \silt \mathcal D$.
    Then $\mu^+_X(S), \mu^-_X(S) \in \silt \mathcal D$.
\end{theorem}

\begin{proposition}[\protect{\cite[Proposition~2.33, Theorem~2.35]{AI}}] \label{properties of mutation}
    Let $S = X \oplus S' \in \silt \mathcal D$.
    \begin{enumerate}
        \item For $S = X \oplus S' \in \silt \mathcal D$, we have
        \[\mu^+_X(S) \geq S \geq \mu^-_X(S).\]
        Moreover, the equalities hold if and only if $X = 0$.
        \item If $X$ is indecomposable, then the previous inequalities are cover relations in $\silt \mathcal D$.
        \item Left and right mutation (with respect to appropriate summands) are mutual inverses (up to isomorphism).
    \end{enumerate}
\end{proposition}

\subsection{Two-term silting complexes and torsion(-free) classes}

In this section, we present a link between torsion(-free) classes in $\modcat \Lambda$
and a particular type of silting objects, called \emph{2-term objects} in the homotopy categories $\Kb(\projcat \Lambda)$ and $\Kb(\injcat \Lambda)$.

We still work in a given triangulated category $\mathcal D$.

\begin{definition} \label{def star notation}
    Let $\mathcal X$, $\mathcal Y$ be two subcategories of $\mathcal D$ such that $\Hom_\mathcal D (\mathcal X, \mathcal Y) = 0$.
    We write $\mathcal X * \mathcal Y$ for the subcategory of $\mathcal D$
    whose objects $Z$ fit in a triangle of the form $X \to Z \to Y \to \Sigma X$
    with $X \in \mathcal X$, $Y \in \mathcal Y$.
    For objects $X$ and $Y$ in $\mathcal D$,
    we write $X * Y$ for the subcategory $(\add X) * (\add Y)$.
\end{definition}

\begin{remark}
    The condition $\Hom_\mathcal D (\mathcal X, \mathcal Y) = 0$ ensures that $\mathcal X * \mathcal Y$ is still a subcategory of $\mathcal D$ under our conventions,
    namely that it is closed under direct summands
    (see for instance \cite[Proposition~2.1]{IYoshino}).
\end{remark}

Now we fix a silting object $S \in \silt \mathcal D$.
Note that by definition we have $\Hom_\mathcal D (S, \Sigma S) = 0$.

\begin{definition}
    An object $X \in \mathcal D$ is said to be \emph{2-term with respect to $S$}
    when $X \in S * \Sigma S$.
\end{definition}

\begin{proposition}[\protect{\cite[Proposition~4.4]{Jasso}}] \label{prop:def 2-term wrt}
    Let $X \in \silt \mathcal D$.
    Then $X \in S * \Sigma S$ if and only if $S \geq X \geq \Sigma S$.
    
    We write $\twosiltwrt S \mathcal D = \silt \mathcal D \cap (S * \Sigma S)$ for the poset of 2-term silting objects with respect to $S$.
    In particular, $\twosiltwrt S \mathcal D$ coincides with the interval $\bigl[\Sigma S, S\bigr]$ inside $\silt \mathcal D$.
\end{proposition}

The interest of considering 2-term silting objects is that they behave better than general silting objects,
in particular with respect to completion.

\begin{proposition} \label{prop:two complements twosilt}
    Let $X$ be a basic presilting object in $\mathcal D$ which is 2-term with respect to $S$.
    \begin{enumerate}
        \item There exists $T \in \twosiltwrt S \mathcal D$ such that $X$ is a direct summand of $T$.
        \item $X$ is silting if and only if $\card X = \card S$,
        where $\card X$ denotes the number of (non-isomorphic) direct summands of $X$.
        \item If $\card X = \card S - 1$,
        then $X$ is a direct summand of exactly two objects in $\twosiltwrt S \mathcal D$.
    \end{enumerate}
\end{proposition}

The three statements are proven in \cite{AIR} for 2-term complexes in $\mathcal D = \Kb(\projcat \Lambda)$ with respect to $\Lambda$,
see \cite[Proposition~3.3]{AIR} for the first two statements and \cite[Corollary~3.8(a)]{AIR} for the third one.
One may translate them to this more general setting using the bijection provided by \cite[Theorem~4.6]{IJY}.
The first two statements are also explicitly stated (and proven) in \cite[Lemma~4.2]{IJY} and \cite[Corollary~4.4]{IJY} respectively.

We now specialise this definition to the case where $\mathcal D$ is either $\Kb(\projcat \Lambda)$ or $\Kb(\injcat \Lambda)$,
where $\Lambda$ is a finite dimensional algebra.

\begin{proposition} \leavevmode
    \begin{enumerate}
        \item The object $\Lambda \in \Kb (\projcat \Lambda)$ is silting,
        and a complex $P \in \Kb (\projcat \Lambda)$ is 2-term with respect to $\Lambda$
        (we also call it a \emph{2-term complex})
        if and only if it is isomorphic in $\Kb (\projcat \Lambda)$ to a complex concentrated in (cohomological) degrees $-1$ and $0$.

        \item The object $\Sigma\inv D\Lambda \in \Kb (\injcat \Lambda)$ is silting,
        and a complex $I \in \Kb (\injcat \Lambda)$ is 2-term with respect to $\Sigma\inv D\Lambda$
        (we also call it a \emph{2-term complex})
        if and only if it is isomorphic in $\Kb (\injcat \Lambda)$ to a complex concentrated in (cohomological) degrees $0$ and $1$.
    \end{enumerate}
\end{proposition}

\begin{convention} \label{convention 2-term complexes}
    Unless otherwise mentioned,
    if we write a complex of projective modules as $(P \to P')$,
    we mean that the complex is concentrated in cohomological degrees $-1$ and 0.
    Similarly, we denote by $(I \to I')$ a complex of injective modules concentrated in degrees 0 and 1.
    Moreover, we identify modules and the corresponding stalk complexes concentrated in degree 0,
    so that for a projective module $P$ (resp.\ an injective module $I$),
    we have
    $P = (0 \to P)$, $\Sigma P = (P \to 0)$,
    $I = (I \to 0)$, $\Sigma\inv I = (0 \to I)$.
\end{convention}

We write $\twosilt(\projcat \Lambda)$ (resp.\ $\twosilt(\injcat \Lambda)$)
for the set of 2-term silting complexes of projective $\Lambda$-modules
(resp.\ 2-term silting complexes of injective $\Lambda$-modules).

The following is a direct corollary of Proposition~\ref{prop:def 2-term wrt}.

\begin{corollary} \label{2-silt proj/inj as interval}
    We have 
    \begin{gather*}
        \twosilt(\projcat \Lambda) = (\silt \Kb(\projcat \Lambda)) \cap (\Lambda * \Sigma \Lambda) = \bigl[\Sigma \Lambda, \Lambda\bigr] \\
        \twosilt(\injcat \Lambda) = (\silt \Kb(\injcat \Lambda)) \cap (\Sigma\inv D \Lambda * D \Lambda) = \bigl[D \Lambda, \Sigma\inv D \Lambda\bigr].
    \end{gather*}
\end{corollary}

\begin{proposition}[\protect{\cite[Theorems 2.7, 3.2]{AIR}}] \label{correspondence twosilt/torsion/torsionfree}
    There is an isomorphism of posets
    \[(\twosilt (\projcat \Lambda), \mathop\leq) \iso (\ftors \Lambda, \mathord{\subseteq})\]
    given by sending a 2-term silting complex of projectives $X$ to the functorially finite torsion class $\Fac H^0(X)$.

    Dually, there is an isomorphism of posets from $(\twosilt (\injcat \Lambda), \leq)$
    to $(\ftorf \Lambda, \mathord{\subseteq})\opcat$ given by sending a 2-term silting complex of injectives $X$ to the functorially finite torsion-free class $\Sub H^0(X)$.
\end{proposition}

\begin{remark}
    The apparent asymmetry between the two statements comes from the fact that for both sets $\ftors \Lambda$ and $\ftorf \Lambda$,
    we have a canonical choice for the partial order, namely inclusion of subcategories $\subseteq$.
    However, for silting objects,
    it is not so clear whether having the relation $\Hom(X, \Sigma^{>0} Y) = 0$ should be interpreted as $X$ being ``greater than'' or ``less than'' $Y$.
    We therefore have to make an arbitrary choice, which in our context is skewed towards torsion classes.
\end{remark}

The Nakayama functor $\nu_\Lambda :\projcat \Lambda \iso \injcat \Lambda$ is an equivalence.
So it extends to a triangle equivalence $\Kb (\projcat \Lambda) \iso \Kb(\injcat \Lambda)$,
which we call the \emph{derived Nakayama functor},
and which we also write $\nu$ by abuse of notation.
Since $\nu$ is a triangle equivalence,
it induces an isomorphism of posets $\silt \Kb(\projcat \Lambda) \iso \silt \Kb(\injcat \Lambda)$.

The following proposition is a translation of \cite[Proposition~2.16]{AIR} in terms of 2-term silting complexes.

\begin{proposition} \label{interpretation Sigma-1 nu}
    We have a commutative diagram of isomorphisms of posets
    \[\begin{tikzcd}[ampersand replacement=\&,cramped]
        {\twosilt(\projcat \Lambda)} \& {\twosilt(\injcat \Lambda)} \\
        {\ftors \Lambda} \& {(\ftorf \Lambda)\opcat}
        \arrow["{\Sigma\inv \nu}", from=1-1, to=1-2]
        \arrow["{\Fac H^0}"', from=1-1, to=2-1]
        \arrow["{\Sub H^0}", from=1-2, to=2-2]
        \arrow["{\blank^\perp}"', from=2-1, to=2-2]
    \end{tikzcd}\]
\end{proposition}

\subsection{1-APR tilting} \label{section 1-APR}

Throughout this section, we fix $\Lambda$ a finite dimensional algebra
admitting a simple projective module $P$,
and we decompose $\Lambda = P \oplus R$ as a $\Lambda$-module.
If $\Lambda = kQ/I$ is the path algebra of a quiver with relations,
then such a module is the indecomposable projective module corresponding to a source $x \in Q_0$.

\begin{lemma} \label{morphisms with target P}
    For $M \in \modcat \Lambda$,
    we have $\Hom_\Lambda(M, P) \neq 0$ if and only if $P$ is (isomorphic to) a direct summand of $M$. 
\end{lemma}

\begin{proof}
    Assume that $\Hom_\Lambda(M,P) \neq 0$,
    and take a non zero morphism $f : M \to P$.
    Since $P$ is simple, $f$ is surjective.
    But $P$ is projective, so $f$ splits,
    and we have $P \in \add M$.
\end{proof}

\begin{lemma} \label{simple projective has to be in socle}
    Let $M \in \modcat \Lambda$.
    Then the simple $P$ does not appear as a composition factor of $M / \soc M$.
    In other words, if $P$ appears as a composition factor of $M$,
    then it appears as a direct summand of $\soc M$.
\end{lemma}

\begin{proof}
    Assume that $P$ is a composition factor of $M$,
    so that there exists some submodule $N \subseteq M$ and a surjection $N \surj P$.
    Now since $P$ is projective, this surjection splits,
    and $P$ is a submodule of $M$.
    Finally, since $P$ is simple, this implies that $P$ is direct summand of $\soc M$.
\end{proof}

\begin{definition} \label{def 1-APR}
    The $\Lambda$-module $T = \tau\inv P \oplus R$ is called the \emph{weak 1-APR tilting module} associated to $P$,
    where $\tau\inv$ denotes the inverse Auslander--Reiten translation of $\modcat \Lambda$.
    If moreover we have $\idim P = 1$, we call $T$ a \emph{1-APR tilting module},
    and the algebra $\Gamma = \End_\Lambda(T)$ is called a \emph{1-APR tilt} of $\Lambda$.
\end{definition}

\begin{remark}
    Dually, we can define the notion of \emph{(weak) 1-APR cotilting module} associated to a simple injective module $I$.
\end{remark}

We fix a minimal injective presentation of $P$:
\[0 \to P \to J_0 \stackrel{a}{\to} J_1.\]

\begin{lemma} \label{inj presentation -> left approximation}
    We have $\nu\inv J_0 = P$, $\nu\inv J_1 \in \add R$,
    and the morphism $P \xrightarrow{\nu\inv a} \nu\inv J_1$
    is an injective minimal left $(\add R)$-approximation of $P$.
\end{lemma}

\begin{proof}
    Since $P$ is simple and the resolution is minimal,
    we have $P \simeq \soc J_0$,
    and $J_0$ is the injective envelope of the simple $P$.
    This implies that $\nu\inv J_0$ is the projective cover of the simple $P$,
    namely $\nu\inv J_0 = P$.

    Moreover,
    to show that $\nu\inv J_1 \in \add R$,
    it suffices to show that the simple $P$ does not appear as a direct summand of $\head (\nu\inv J_1) \simeq \soc J_1$.
    Now we know that $\soc J_1 = \soc (J_0 / P)$ since the presentation is minimal,
    and that $J_0 / P \simeq J_0 / \soc J_0$.
    So by Lemma~\ref{simple projective has to be in socle},
    the simple $P$ does not appear as a composition factor of $J_0 / P$,
    hence the result.
    
    Now we show that $\nu \inv a$ is a left $(\add R)$-approximation of $P$.
    Let $f : P \to S$ be a morphism with $S \in \add R$.
    This gives a morphism $\nu (f) : J_0 \to \nu S$.
    Now the composition $P \to J_0 \xrightarrow{\nu (f)} \nu S$ has to be zero since $\Hom(P, \nu R) = 0$.
    So $\nu (f)$ factors through the cokernel $C$ of $P \to J_0$,
    hence it factors through $a$ since $\nu S$ is an injective module.
    So we get a morphism $g : J_1 \to \nu S$ satisfying $ga = \nu (f)$,
    hence the morphism $\nu\inv (g) : \nu\inv J_1 \to S$ satisfies $\nu\inv (g) \circ \nu\inv (a) = f$.
    This shows what we wanted.
    \[\begin{tikzcd}
        P & {J_0} & C & {J_1} \\
        & {\nu S}
        \arrow[from=1-1, to=1-2]
        \arrow[two heads, from=1-2, to=1-3]
        \arrow["a", curve={height=-18pt}, from=1-2, to=1-4]
        \arrow["{\nu (f)}"', from=1-2, to=2-2]
        \arrow[hook, from=1-3, to=1-4]
        \arrow[dashed, from=1-3, to=2-2]
        \arrow["g", curve={height=-12pt}, dashed, from=1-4, to=2-2]
    \end{tikzcd}\]

    Finally, the minimality of the approximation follows from the minimality of the injective resolution,
    and $\nu\inv a$ is injective because it is non zero and $P$ is simple.
\end{proof}

\begin{lemma} \label{proj presentation of tau- P}
    We have an exact sequence
    \[0 \to P \xrightarrow{\nu\inv a} \nu\inv J_1 \to \tau\inv P \to 0\]
    which gives a minimal projective resolution of $\tau\inv P$.
\end{lemma}

\begin{proof}
    This follows from the definition of the functor $\tau\inv$ and from Lemma~\ref{inj presentation -> left approximation}.
\end{proof}

\begin{theorem}[\protect{\cite[Theorem~3.2]{IO}}] \label{1-APR implies tilting}
    The weak 1-APR tilting module $T$ is a tilting module with $\pdim T = 1$.
\end{theorem}

\begin{corollary} \label{1-APR tilting derived equivalence}
    The functor $\mathbf R \Hom_\Lambda(T, \blank) : \deriv (\Lambda) \to \deriv (\Gamma)$
    is an equivalence of triangulated categories.
\end{corollary}

\begin{proof}
    The fact that a tilting module induces a derived equivalence is a theorem of Happel,
    see \cite[{}III.2.10]{Happel}.
\end{proof}

We consider the functor $F = \Hom_\Lambda(T, \blank) : \modcat \Lambda \to \modcat \Gamma$,
where $\Gamma = \End_\Lambda (T)$.
We know (see for instance \cite[{}VI.3.1]{ASS}) that $F$ restricts to an equivalence $\add T \iso \add \Gamma$.

\begin{proposition}[\protect{\cite[Lemma~3.7(2)]{IO}}] \label{simple injective Gamma module}
    The $\Gamma$-module $I = \Ext^1_\Lambda(T,P)$ is simple,
    and we have an exact sequence
    \[0 \to F(\nu\inv J_1) \stackrel{b}{\to} F(\tau\inv P) \to I \to 0\]
    in $\modcat \Gamma$ which gives a projective resolution of $I$.
    In particular, we have $\pdim I = 1$.

    If moreover $T$ is a 1-APR tilting module, then $I$ is injective.
\end{proposition}

\begin{proof}
    We only show the existence of the exact sequence,
    and we refer to \cite{IO} for the proof of the remaining claims.

    We apply the functor $F = \Hom_\Lambda(T,\blank)$ to the exact sequence
    \[0 \to P \xrightarrow{\nu\inv a} \nu\inv J_1 \to \tau\inv P \to 0\]
    from Lemma~\ref{proj presentation of tau- P} to get an exact sequence
    \[0 \to FP \to F(\nu\inv J_1) \to F(\tau\inv P) \to \Ext^1_\Lambda(T,P) \to \Ext^1_\Lambda(T,\nu\inv J_1).\]
    Now we know that $FP = \Hom_\Lambda(T,P) = 0$ since $P$ is not in $\add T$ (see \ref{morphisms with target P}),
    and we have $\Ext^1_\Lambda(T, \nu\inv J_1) = 0$ since $\nu\inv J_1 \in \add R \subseteq \add T$ and $T$ is tilting.
    So we get the desired short exact sequence,
    and it is a projective resolution of $I$ since both $F(\nu\inv J_1)$ and $F(\tau\inv P)$
    belong to the subcategory $F(\add T) = \add \Gamma = \projcat \Gamma$.
\end{proof}

Until the end of this section,
we assume that $\idim P = 1$,
that is, $T$ is a 1-APR tilting module.

\begin{lemma} \label{nu F(R) = J}
    Set $I = \Ext^1_\Lambda(T,P)$, which is a simple injective $\Gamma$-module by Proposition~\ref{simple injective Gamma module},
    and decompose $D\Gamma = I \oplus J$.
    Then we have $\nu_\Gamma F(\tau\inv P) \simeq I$ and $\nu_\Gamma F(R) \simeq J$.
\end{lemma}

\begin{proof}
    We know by Proposition~\ref{simple injective Gamma module} that the top of the projective $\Gamma$-module $F (\tau\inv P)$
    is isomorphic to $I$,
    which implies that $\nu_\Gamma (F\tau\inv P) = I$.
    Moreover, we know that
    \[\nu_\Gamma (FT) = \nu_\Gamma F(\tau\inv P) \oplus \nu_\Gamma F(R) = D\Gamma = I \oplus J,\]
    so by the Krull--Remak--Schmidt theorem we conclude that $\nu_\Gamma F(R) = J$.
\end{proof}

\begin{lemma} \label{nu and F commute}
    We have an isomorphism $\nu_\Gamma F \simeq F \nu_\Lambda$
    of functors $\add R \iso \add J$.
\end{lemma}

\begin{proof}
    Let $X \in \add R$,
    so $X$ is both projective and in $\add T$.
    We have the following isomorphisms of $\Gamma$-modules,
    which are natural in $X$:
    \begin{align*}
        F \circ \nu_\Lambda (X) &= \Hom_\Lambda (T, D \Hom_\Lambda (X, \Lambda)) \\
        &\simeq D (T \otimes_\Lambda \Hom_\Lambda (X, \Lambda))
    \end{align*}
    and
    \begin{align*}
        \nu_\Gamma \circ F (X) &= D \Hom_\Gamma (F(X), \Gamma) \\
        &= D \Hom_\Gamma (F(X), F(T)) \\
        &\simeq D \Hom_\Lambda (X, T)
    \end{align*}
    where the last isomorphism holds because $X \in \add T$ and the functor $F$ restricts to an equivalence $\add T \iso \add \Gamma$.

    Now the result follows from the natural isomorphism
    $T \otimes_\Lambda \Hom_\Lambda (X, \Lambda) \simeq \Hom_\Lambda (X,T)$,
    which holds because $X$ is a projective $\Lambda$-module.
\end{proof}

\begin{lemma} \label{nu b is approx}
    The morphism $FJ_1 \xrightarrow{\nu b} I$
    is a surjective minimal right $(\add J)$-approximation of~$I$.
\end{lemma}

\begin{proof}
    First, recall that $b$ is a morphism $F(\nu\inv J_1) \xrightarrow{b} F(\tau\inv P)$ of $\Gamma$-modules
    which was introduced in Proposition~\ref{simple injective Gamma module}.
    Now the source of $\nu b$ is $\nu_\Gamma \circ F(\nu_\Lambda\inv J_1)$,
    which by Lemma~\ref{nu and F commute} is isomorphic to $F(\nu_\Lambda \circ \nu_\Lambda\inv J_1) \simeq F(J_1)$.
    Similarly, the target of $\nu b$ is $\nu_\Gamma \circ F(\tau\inv P)$,
    which by Lemma~\ref{nu F(R) = J} is isomorphic to $I$.
    This ensures that the morphism $F J_1 \xrightarrow{\nu b} I$ in the statement of the lemma is well-defined.
    The proof is then analogous to that of Lemma~\ref{inj presentation -> left approximation}.
\end{proof}

\subsection{Flip-flop and derived equivalence of posets} \label{sec:prelim flip-flop}

The notion of flip-flop was introduced by Ladkani in \cite{Lad-universal}.
It is a purely combinatorial operation on posets,
which yields a derived equivalence of their incidence algebras.

\begin{definition}
    Let $(X,\leq_X)$ and $(Y,\leq_Y)$ be posets,
    let $f : X \to Y$ be an order-preserving map.
    We define two partial orders on the set $X \sqcup Y$,
    called $\leq^f_+$ and $\leq^f_-$, as follows:
    \begin{enumerate}
        \item for $x, x' \in X$, we have $x \leq^f_+ x' \iff x \leq^f_- x' \iff x \leq_X x'$
        \item for $y, y' \in Y$, we have $y \leq^f_+ y' \iff y \leq^f_- y' \iff y \leq_Y y'$
        \item for $x \in X$ and $y \in Y$, we have $x \leq^f_+ y \iff f(x) \leq_Y y$ and $y \leq^f_- x \iff y \leq_Y f(x)$,
    \end{enumerate}
    these being the only relations in the posets $(X \sqcup Y, \leq^f_+)$ and $(X \sqcup Y, \leq^f_-)$.

    Moreover, we say that two posets $(Z, \leq_Z)$ and $(Z', \leq_{Z'})$
    are \emph{related by a flip-flop} when there exist $X$, $Y$ and $f$ such that
    $(Z, \leq_Z) \simeq (X \sqcup Y, \leq^f_+)$ and $(Z', \leq_{Z'}) \simeq (X \sqcup Y, \leq^f_-)$.
    We also say that $Z$ and $Z'$ each admits a \emph{flip-flop decomposition}.
\end{definition}

\begin{example}
    Figure \ref{figure intro} provides an example of flip-flop,
    where $X = \{x_0 < x_1\}$ (in gold), \mbox{$Y = \{y_0 < y_1 < y_2\}$} (in silver),
    and the map $f : X \to Y$ is defined by $f(x_0) = y_0$ and $f(x_1) = y_2$.
    The poset on the left is then $(X \sqcup Y, \leq^f_-)$ and the one on the right is $(X \sqcup Y, \leq^f_+)$.
\end{example}

The motivation for defining this notion rests in the following theorem,
which is a particular case of Corollary~1.3 in \cite{Lad-universal}:

\begin{theorem}
    Let $Z$ and $Z'$ be two posets related by a flip-flop.
    Then for any abelian category $\mathcal A$,
    the categories of functors $\Fun (Z\opcat, \mathcal A)$ and $\Fun ({Z'}\opcat, \mathcal A)$
    are derived equivalent.
\end{theorem}

If we consider finite posets and take $\mathcal A$ to be the category of finite dimensional $k$-vector spaces,
we get the following corollary.

\begin{corollary}
    Let $Z$ and $Z'$ be two finite posets related by a flip-flop.
    Then their incidence algebras $kZ$ and $kZ'$ are derived equivalent,
    that is we have $\deriv (kZ) \simeq \deriv (kZ')$.
\end{corollary}

When the order-preserving map $f$ defining a flip-flop structure is injective,
the structure of the Hasse quivers of the posets and the effect of the flip-flop on them is easily described.

\begin{proposition}
    Assume that $f : X \to Y$ is injective.
    Then the Hasse quiver $H_Z$ of the poset $(X \sqcup Y, \leq^f_+)$
    (resp.\ $H_{Z'}$ of $(X \sqcup Y, \leq^f_-)$)
    is obtained by gluing the Hasse quivers $H_X$ and $H_Y$ of $X$ and $Y$
    along the new arrows $(f(x) \to x)$
    (resp.\ $(x \to f(x))$)
    for $x \in X$.

    In particular, the quivers $H_Z$ and $H_{Z'}$ have the same underlying graph and only differ by flipping the arrows between $x$ and $f(x)$ for $x \in X$.
\end{proposition}

\begin{proof}
    We show the result only for $Z = (X \sqcup Y, \leq^f_+)$,
    the other case being analogous.
    It is clear that the covering relations (corresponding to arrows of the Hasse quiver)
    in the posets $X$ and $Y$ remain covering relations when embedded in $Z$
    (since the only relations between an element $x$ of $X$ and an element $y$ of $Y$ are of the form $x < y$).
    So it suffices to show that the relations $x < f(x)$ for $x \in X$ are covering relations in $Z$.

    Let $x \in X$.
    We assume by contradiction that there exists some $z \in Z$ such that $x < z < f(x)$.
    First assume that $z \in X$.
    Then by the definition of $\leq^f_+$ this implies that $f(z) \leq f(x)$,
    hence $f(z) < f(x)$ since $f$ is injective and $x \neq z$.
    But this contradicts the fact that $x < z$ since $f$ is order-preserving.

    Now assume that $z \in Y$.
    Then by the definition of $\leq^f_+$ this implies that $f(x) \leq z$,
    which contradicts the assumption that $z < f(x)$. This concludes the proof.
\end{proof}

The following lemma explains what happens when considering the opposite of a poset constructed by a flip-flop.

\begin{lemma} \label{opposite flip-flop}
    Let $f : X \to Y$ be an order-preserving map.
    Then
    \[\bigl(X \sqcup Y, \leq^f_-\bigr)\opcat \simeq \bigl(X\opcat \sqcup Y\opcat, \leq^{f\opcat}_+\bigr)\]
    where $f\opcat : X\opcat \to Y\opcat$ is the induced order-preserving map.
\end{lemma}

The following proposition is helpful to show that a poset admits a flip-flop decomposition.

\begin{proposition} \label{characterisation of flip-flop decomposition}
    Let $(Z, \leq_Z)$ be a poset,
    and assume that there is a decomposition $Z = X \sqcup Y$ (as sets) and a map $f : X \to Y$
    (which we do not assume to be order-preserving)
    such that the following conditions hold:
    for every $x \in X$, $y \in Y$, we have $x \leq y \iff f(x) \leq y$
    (resp.\ $y \leq x \iff y \leq f(x)$)
    and $y \nleq x$ (resp.\ $x \nleq y$).
    
    Then $f$ is order-preserving,
    and we have $(Z, \leq_Z) \simeq (X \sqcup Y, \leq^f_-)$
    (resp.\ $(Z, \leq_Z) \simeq (X \sqcup Y, \leq^f_+)$).
\end{proposition}

\begin{proof}
    We show the result only for $Z = (X \sqcup Y, \leq^f_+)$,
    the other case being analogous.
    First, notice that the assumption implies that for every $x \in X$,
    we have $x \leq f(x)$.
    We show that $f$ is order-preserving:
    let $x \leq x' \in X$.
    We have $x \leq x' \leq f(x')$,
    so by the assumption this implies $f(x) \leq f(x')$,
    so $f$ is order-preserving.
    It is then easy to check that the partial order $\leq_Z$ coincides with the one defined by the flip-flop data.
\end{proof}

\hidepagenumber
\section{Flip-flop for functorially finite torsion classes}
\showpagenumber

\subsection{The flip-flop decomposition} \label{torsion classes wrt source}

In the following, we consider an algebra $\Lambda$ which has a simple projective module $P$.
If $\Lambda = kQ/I$ is the path algebra of a quiver with relations,
then such a module is the indecomposable projective module corresponding to a source in the quiver $Q$.

\begin{lemma} \label{simple projective in Fac}
    Let $M \in \modcat \Lambda$ such that $P \in \Fac M$.
    Then $P$ is a direct summand of $M$.
\end{lemma}

\begin{proof}
    This is a direct consequence of Lemma~\ref{morphisms with target P}.
\end{proof}

\begin{lemma} \label{removing P from Fac M}
    Let $M = P \oplus M'$ be a basic $\Lambda$-module.
    Then we have $\ind \Fac M' = (\ind \Fac M) \setminus \{P\}$.
\end{lemma}

\begin{proof}
    First, we clearly have $\Fac M' \subseteq \Fac M$,
    and $P \notin \Fac M'$ by Lemma~\ref{simple projective in Fac},
    so we get the inclusion $\ind \Fac M' \subseteq (\ind \Fac M) \setminus \{P\}$.

    It remains to show that every indecomposable module in $\Fac M$ except $P$ is in $\Fac M'$.
    Let $X$ be such a module.
    We know that there exists an epimorphism $h : M^n \surj X$ for some integer $n \geq 1$.
    Decompose $h = (h_1, h_2)$,
    where $h_1 : P \to X$, $h_2 : M'' \to X$ and $M^n = P \oplus M''$,
    and assume that $h_1 \neq 0$.
    Since $P$ is simple, $h_1$ is injective,
    and we have $P \simeq \im h_1 \subseteq \soc X$.
    Since $X$ is indecomposable and not isomorphic to $P$,
    this implies that $\im h_1 \subseteq \rad X$.
    So $h_2$ is surjective:
    indeed, the contrary would imply $\im h_2 \subseteq \rad X$,
    hence $\im h = \im h_1 + \im h_2 \subseteq \rad X$,
    and this would contradict the surjectivity of $h$.

    Now by induction on $n$,
    we conclude that there exists an epimorphism $h' : {M'}^n \surj X$,
    which concludes the proof.
\end{proof}

\begin{lemma} \label{T1 is a torsion class}
    The subcategory $\leftperp P \subseteq \modcat \Lambda$ is a functorially finite torsion class,
    and we have $\ind (\leftperp P) = (\ind \modcat \Lambda) \setminus \{P\}$.
\end{lemma}

\begin{proof}
    We know that $\add P$ is closed under submodules and extensions,
    since every module in $\add P$ is semi-simple and $P$ is projective.
    So $\add P$ is a torsion-free class by Proposition~\ref{prop:characterisation of torsion classes},
    and it is functorially finite since it only contains one indecomposable module.
    So by Proposition~\ref{prop:equivalence ff torsion(free)} we get that
    $\leftperp P = \leftperp{(\add P)}$ is a functorially finite torsion class.

    Now by Lemma~\ref{morphisms with target P},
    we see that an indecomposable module $M$ satisfies $\Hom_\Lambda(M,P) = 0$ if and only if it is not isomorphic to $P$,
    which shows that $\ind (\leftperp P) = (\ind \modcat \Lambda) \setminus \{P\}$.
\end{proof}

We consider the subset $\tors_P \Lambda = \{\mathcal T \in \tors \Lambda \mid P \in \mathcal T\}$ of $\tors \Lambda$.
Then we define a map $f : \tors_P \Lambda \to \tors \Lambda \setminus \tors_P \Lambda$
by setting $f(\mathcal T) = \mathcal T \cap \leftperp P$,
for a torsion class $\mathcal T \in \tors_P \Lambda$.
In other words, the torsion class $f(\mathcal T)$ is given by
$\ind f(\mathcal T) = (\ind \mathcal T) \setminus \{P\}$.

\begin{proposition} \label{mutation for torsion classes}
    The map $f$ is well-defined as a map $\tors_P \Lambda \to \tors \Lambda \setminus \tors_P \Lambda$.

    Moreover, if $\mathcal T = \Fac M \in \tors_P \Lambda$ with $M = P \oplus M'$ a basic module,
    then $f(\mathcal T) = \Fac M'$.

    So $f$ restricts to a map $\ftors_P \Lambda \to \ftors \Lambda \setminus \ftors_P \Lambda$,
    which we still call $f$,
    where we have set $\ftors_P \Lambda = \tors_P \Lambda \cap \ftors \Lambda$.
\end{proposition}

\begin{proof}
    We first check that $f$ is well-defined.
    Let $\mathcal T \in \tors_P \Lambda$.
    The subcategory $f(\mathcal T)$ is still a torsion class,
    since it is defined as the intersection of two torsion classes.
    Moreover, it is clear that $P \notin f(\mathcal T)$,
    so that $f(\mathcal T) \in \tors \Lambda \setminus \tors_P \Lambda$.

    The second claim follows directly from Lemma~\ref{removing P from Fac M} and the definition of $f$.
\end{proof}

\begin{proposition} \label{tors has flip-flop decomposition}
    The poset $\tors \Lambda$ has a flip-flop decomposition given by
    \[(\tors \Lambda, \mathord{\subseteq}) \simeq \bigl(\tors_P \Lambda \sqcup (\tors \Lambda \setminus \tors_P \Lambda), \leq^f_-\bigr)\]
    which restricts to a flip-flop decomposition of $\ftors \Lambda$,
    namely we have
    \[(\ftors \Lambda, \mathord{\subseteq}) \simeq \bigl(\ftors_P \Lambda \sqcup (\ftors \Lambda \setminus \ftors_P \Lambda), \leq^f_-\bigr).\]
\end{proposition}

\begin{proof}
    By Proposition~\ref{characterisation of flip-flop decomposition},
    it suffices to show that for $\mathcal T \in \tors_P \Lambda$ and $\mathcal T' \in \tors \Lambda \setminus \tors_P \Lambda$,
    we have $\mathcal T' \subseteq \mathcal T$ if and only if $\mathcal T' \subseteq f(\mathcal T)$,
    as well as $\mathcal T \nsubseteq \mathcal T'$.

    The latter claim follows from the fact that $P \in \mathcal T$ and $P \notin \mathcal T'$.
    For the former claim, we have the equivalence
    \[\mathcal T' \subseteq \mathcal T \iff \ind \mathcal T' \subseteq \ind \mathcal T\]
    and since $\ind f(\mathcal T) = (\ind \mathcal T) \setminus \{P\}$ and $P \notin \ind \mathcal T'$,
    clearly we get the equivalence
    \[\mathcal T' \subseteq \mathcal T \iff \mathcal T' \subseteq f(\mathcal T).\qedhere\]
\end{proof}

\begin{proposition} \label{flip-flop intervals tors}
    Both subsets
    appearing in the flip-flop decomposition of $\tors \Lambda$ are intervals of $\tors \Lambda$.
    Namely, we have
    $\tors_P \Lambda = \bigl[\add P, \modcat \Lambda\bigr]$ and
    $\tors \Lambda \setminus \tors_P \Lambda = \bigl[\{0\}, \leftperp P\bigr]$.
    
    Moreover, we have the same description for $\ftors_P \Lambda$ and $\ftors \Lambda \setminus \ftors_P \Lambda$
    as intervals of $\ftors \Lambda$.
\end{proposition}

\begin{proof}
    First, we know that $\leftperp P$ is a (functorially finite) torsion class by Lemma~\ref{T1 is a torsion class},
    and $\add P$ is also a functorially finite torsion class (since $P$ is simple).

    Now it suffices to show that for $\mathcal T \in \tors \Lambda$,
    we have the equivalences
    \[\mathcal T \in \tors_P \Lambda \iff \add P \subseteq \mathcal T\]
    and
    \[\mathcal T \notin \tors_P \Lambda \iff \mathcal T \subseteq \leftperp P.\]
    Both these claims are clear from the definition of $\tors_P \Lambda$
    and the fact that $\ind \add P = \{P\}$.
\end{proof}

We show that the torsion-free class $\rightperp P$ associated to the torsion class $\add P$ is in fact a Serre subcategory of $\modcat \Lambda$.

\begin{lemma} \label{F is abelian}
    The subcategory $\rightperp P \subseteq \modcat \Lambda$ is a Serre subcategory.
    In particular, it is abelian.
\end{lemma}

\begin{proof}
    We want to show that for every short exact sequence $0 \to L \to M \to N \to 0$ in $\modcat \Lambda$,
    we have
    \[M \in \rightperp P \iff L \in \rightperp P \textnormal{ and } N \in \rightperp P.\]
    
    Now applying the functor $\Hom_\Lambda (P, \blank)$ to such an exact sequence yields an exact sequence
    \[0 \to \Hom_\Lambda(P,L) \to \Hom_\Lambda(P,M) \to \Hom_\Lambda(P,N) \to 0\]
    since $P$ is projective,
    and the desired equivalence then follows.
\end{proof}

\begin{remark} \label{Jasso reduction}
    The subset $\tors_P \Lambda$ can be further understood in the framework of reduction introduced by Jasso in \cite{Jasso}.
    We consider the support $\tau$-tilting module $P$, whose Bongartz completion is $\Lambda$.
    We set $\Delta = \Lambda / \langle e_P \rangle$,
    that is, $\Delta$ is the algebra obtained by taking the quotient of $\Lambda$ by the two-sided ideal generated by the primitive idempotent $e_P$ corresponding to $P$.
    Under this setup, applying \cite[Theorem~3.12]{Jasso} gives an isomorphism of posets
    \[\tors_P \Lambda \simeq \tors \Delta.\]

    Moreover, if $\Lambda = kQ/I$ is the path algebra of a quiver with relations and $P$ corresponds to a source $x \in Q_0$,
    we have $\idim P = 1$ if and only if the ideal $I$ contains no relations starting at the vertex $x$.
    In this case, we have $\Delta \simeq k (Q \setminus \{x\}) / I$.
\end{remark}

We now want to transpose these results to a dual setting,
namely investigating the structure of the poset $(\torf \Gamma, \subseteq)$ of torsion-free classes of an algebra $\Gamma$ which admits a simple injective module $I$.
With the appropriate definitions, all the statements and the proofs transpose directly,
so we simply record them for convenience.

\begin{lemma} \label{F1 is a torsion-free class}
    The subcategory $\rightperp I \subseteq \modcat \Gamma$ is a functorially finite torsion-free class,
    and we have $\ind (\rightperp I) = (\ind \modcat \Gamma) \setminus \{I\}$.
\end{lemma}

We consider the subset $\torf_I \Gamma = \{\mathcal F \in \torf \Gamma \mid I \in \mathcal F\}$ of $\torf \Gamma$.
Then we define a map $g : \torf_I \Gamma \to \torf \Gamma \setminus \torf_I \Gamma$
by setting $g(\mathcal F) = \mathcal F \cap \rightperp I$,
for a torsion-free class $\mathcal F \in \torf_I \Gamma$.
In other words, the torsion-free class $g(\mathcal F)$ is given by
$\ind g(\mathcal F) = (\ind \mathcal F) \setminus \{I\}$.

\begin{proposition} \label{torf has flip-flop decomposition}
    The poset $\torf \Gamma$ has a flip-flop decomposition given by
    \[(\torf \Gamma, \mathord{\subseteq}) \simeq \bigl(\torf_I \Gamma \sqcup (\torf \Gamma \setminus \torf_I \Gamma), \leq^g_-\bigr),\]
    and this restricts to a flip-flop decomposition of $\ftorf \Gamma$,
    namely we have
    \[(\ftorf \Gamma, \mathord{\subseteq}) \simeq \bigl(\ftorf_I \Gamma \sqcup (\ftorf \Gamma \setminus \ftorf_I \Gamma), \leq^g_-\bigr).\]
\end{proposition}

\begin{proposition} \label{flip-flop intervals torf}
    Both subsets
    appearing in the flip-flop decomposition of $\torf \Gamma$ are intervals of $\torf \Gamma$.
    Namely, we have
    $\torf_I \Gamma = \bigl[\add I, \modcat \Gamma\bigr]$ and
    $\torf \Gamma \setminus \torf_I \Gamma = \bigl[\{0\}, \rightperp I\bigr]$.
    
    Moreover, we have the same description for $\ftorf_I \Gamma$ and $\ftorf \Gamma \setminus \ftorf_I \Gamma$
    as intervals of $\ftorf \Gamma$.
\end{proposition}

\begin{lemma} \label{T' is abelian} 
    The subcategory $\leftperp I \subseteq \modcat \Gamma$ is a Serre subcategory.
    In particular, it is abelian.
\end{lemma}

\subsection{Translation in terms of 2-term silting complexes}

Our goal is to translate the results of section \ref{torsion classes wrt source}
in the language of 2-term silting complexes.
As in section \ref{torsion classes wrt source}, we consider an algebra $\Lambda$ which admits a simple projective module $P$,
and we write $\Lambda = P \oplus R$.
We fix a minimal injective presentation
\[0 \to P \to J_0 \stackrel{a}{\to} J_1\]
of $P$, as in section \ref{section 1-APR}.
By Lemma~\ref{inj presentation -> left approximation},
we have $\nu\inv J_0 = P$,
and the morphism $\nu\inv a$ is a minimal left $(\add R)$-approximation of $P$.
We write $S = (P \xrightarrow{\nu\inv a} \nu\inv J_1) \in \Kb(\projcat \Lambda)$.
We also set $\widetilde T = S \oplus R \in \Kb (\projcat \Lambda)$.

\begin{remark} \label{S is projective resolution of tau-P}
    The complex $S$ has cohomology concentrated in degree 0,
    and $H^0(S) = \tau\inv P$, according to Lemma~\ref{proj presentation of tau- P}.
    Thus the 2-term complex $S$ is isomorphic, in $\deriv (\Lambda)$, to the stalk complex $\tau\inv P$.
    This implies that the weak 1-APR tilting module $T = \tau\inv P \oplus R$ is isomorphic, in $\deriv (\Lambda)$, to $\widetilde T \coloneqq S \oplus R$,
    which we view as an element of $\twosilt (\projcat \Lambda)$.
\end{remark}

The first step is to translate the flip-flop decomposition of $\ftors \Lambda$
into a similar decomposition for the poset $\twosilt (\projcat \Lambda)$.
To this end, we have to understand what the subposets $\ftors_P \Lambda$ and $\ftors \Lambda \setminus \ftors_P \Lambda$
and the map $f : \ftors_P \Lambda \to \ftors \Lambda \setminus \ftors_P \Lambda$ correspond to under the bijection between $\ftors \Lambda$ and $\twosilt (\projcat \Lambda)$.

\begin{proposition} \label{silt_P maps to tors_P}
    Let $X \in \twosilt (\projcat \Lambda)$.
    Then $\Fac H^0 (X) \in \ftors_P \Lambda$ if and only if $P \in \add X$.
\end{proposition}

\begin{proof}
    We know by Proposition~\ref{correspondence twosilt/torsion/torsionfree} that $\Fac H^0(X) \in \ftors \Lambda$.
    Moreover, by Lemma~\ref{simple projective in Fac},
    we have $\Fac H^0(X) \in \ftors_P \Lambda$ if and only if $P \in \add H^0(X)$.
    Now the equivalence
    \[P \in \add H^0(X) \iff P \in \add X\]
    follows from the fact that $P \in \Kb (\projcat \Lambda)$
    (seen as a complex concentrated in degree 0)
    is a 2-term presilting complex,
    and that it is the only indecomposable 2-term complex whose cohomology in degree zero is $P$.
\end{proof}

We define the subset $\twosilt_P(\projcat \Lambda) = \{X \in \twosilt(\projcat \Lambda) \mid P \in \add X\}$.

\begin{proposition} \label{correspondence 2-silt_P and tors_P}
    We have a commutative diagram
    \[\begin{tikzcd}[ampersand replacement=\&,cramped]
        {\twosilt_{P} (\projcat \Lambda)} \& {\twosilt (\projcat \Lambda) \setminus \twosilt_{P} (\projcat \Lambda)} \\
        {\ftors_P \Lambda} \& {\ftors \Lambda \setminus \ftors_P \Lambda}
        \arrow["{\mu^-_{P}}", from=1-1, to=1-2]
        \arrow["{\Fac H^0}"', from=1-1, to=2-1]
        \arrow["{\Fac H^0}", from=1-2, to=2-2]
        \arrow["f"', from=2-1, to=2-2]
    \end{tikzcd}\]
    where the vertical maps are isomorphisms of posets.
    
    In particular, this implies that the map
    $\mu^-_{P} : \twosilt_{P} (\projcat \Lambda) \to \twosilt (\projcat \Lambda) \setminus \twosilt_{P} (\projcat \Lambda)$
    is order-preserving.
\end{proposition}

\begin{proof}
    First, we see that Proposition~\ref{silt_P maps to tors_P} and Proposition~\ref{correspondence twosilt/torsion/torsionfree} ensure that both vertical maps are well-defined.

    Now we show that the diagram commutes.
    Let $X = P \oplus X' \in \twosilt_{P} (\projcat \Lambda)$.
    Take $P \to Y$ a minimal left $(\add X')$-approximation of $P$,
    and consider the triangle $P \to Y \to Z \to \Sigma P$ in $\Kb (\projcat \Lambda)$,
    so that we have $\mu^-_{P}(X) = Z \oplus X'$.
    The fact that $Z$ is an extension of two 2-term complexes
    (namely $Y$ and $\Sigma P$)
    ensures that it is 2-term.
    This implies that $\mu^-_P(X)$ belongs to $\twosilt(\projcat \Lambda) \setminus \twosilt_P(\projcat \Lambda)$.
    Now applying the cohomological functor $H^0$ to this triangle,
    we get an exact sequence $H^0(Y) \to H^0(Z) \to H^0(\Sigma P) = 0$ in $\modcat \Lambda$,
    so we get an epimorphism $H^0(Y) \surj H^0(Z)$.
    This implies that $\Fac H^0(Z) \subseteq \Fac H^0(Y) \subseteq \Fac H^0(X')$,
    so finally we obtain
    $\Fac H^0(\mu^-_{P}(X)) = \Fac (H^0(Z) \oplus H^0(X')) = \Fac H^0(X')$.

    Moreover, we have $\Fac H^0(X) = \Fac (P \oplus H^0(X'))$,
    so according to Proposition~\ref{mutation for torsion classes},
    we get $f(\Fac H^0(X)) = \Fac H^0(X')$.
    This shows that the diagram commutes.
\end{proof}

The following is a direct corollary of Proposition~\ref{correspondence 2-silt_P and tors_P} using Proposition~\ref{tors has flip-flop decomposition}.

\begin{corollary} \label{flip-flop for 2-silt proj}
    The poset $\twosilt(\projcat \Lambda)$ has a flip-flop decomposition:
    \[\twosilt(\projcat \Lambda) \simeq \Bigl(\twosilt_P(\projcat \Lambda) \sqcup \twosilt(\projcat \Lambda) \setminus \twosilt_P(\projcat \Lambda), \leq^{\mu_P^-}_-\Bigr).\]
\end{corollary}

We now translate the description of $\ftors_P \Lambda$ and $\ftors \Lambda \setminus \ftors_P \Lambda$ as intervals of $\ftors \Lambda$ given in Proposition~\ref{flip-flop intervals tors}
to the context of silting complexes.
It follows immediately that both subsets
$\twosilt_P (\projcat \Lambda)$ and $\twosilt (\projcat \Lambda) \setminus \twosilt_P (\projcat \Lambda)$ are intervals of $\twosilt (\projcat \Lambda)$.
We can moreover compute the bounds of these intervals explicitly.

\begin{proposition} \label{intervals for 2-silt proj}
    We have
    \begin{gather*}
        \twosilt_P (\projcat \Lambda) = \bigl[P \oplus \Sigma R, \Lambda\bigr] \\
        \twosilt (\projcat \Lambda) \setminus \twosilt_P (\projcat \Lambda) = \bigl[\Sigma \Lambda, \widetilde T\bigr].
    \end{gather*} 
\end{proposition}

\begin{proof}
    First, it is clear that $\Fac H^0(\Lambda) = \modcat \Lambda$ and
    $\Fac H^0(\Sigma \Lambda) = \{0\}$.
    Moreover, it is easy to check that $P \oplus \Sigma R$ is silting,
    and it is clearly a 2-term complex whose cohomology in degree 0 is $P$,
    so that $\Fac H^0(P \oplus \Sigma R) = \add P$.

    It remains to compute the upper bound $Z$ of the interval $\twosilt (\projcat \Lambda) \setminus \twosilt_P (\projcat \Lambda)$,
    that is the 2-term silting complex satisfying $\Fac H^0(Z) = \leftperp P$.
    We do so by using Proposition~\ref{correspondence 2-silt_P and tors_P}:
    since we know that $f(\modcat \Lambda) = \leftperp P$,
    we get $Z = \mu^-_P(\Lambda) = \mu^-_P (P \oplus R)$.
    Now we know that $P \xrightarrow{\nu\inv a} \nu\inv J_1$ is a minimal left $(\add R)$-approximation of $P$.
    Then the 2-term complex \mbox{$S = (P \xrightarrow{\nu\inv a} \nu\inv J_1)$} is the cone of the morphism $\nu\inv a$
    in $\Kb (\projcat \Lambda)$,
    hence by the definition of mutation we get $Z = S \oplus R = \widetilde T$,
    as desired.
\end{proof}

As in section \ref{torsion classes wrt source},
we now want to translate these results to the dual setting where we replace the algebra $\Lambda$ with a simple projective module $P$
by an algebra $\Gamma$ with a simple injective module $I$.
As before, we write $D\Gamma = I \oplus J$.
We take $J' \stackrel{c}{\to} I$ a minimal right $(\add J)$-approximation of $I$,
and we define a 2-term complex $S' = (J' \stackrel{c}{\to} I) \in \Kb(\injcat \Gamma)$.

We have to be careful to the fact that the correspondence between 2-term silting complexes of injectives and torsion-free classes is given by
and anti-isomorphism between the posets $\twosilt (\injcat \Gamma)$ and $(\ftorf \Gamma, \subseteq)$.
Namely, we have $\twosilt (\injcat \Gamma) \simeq (\ftorf \Gamma, \subseteq)\opcat$.

\begin{proposition}
    Let $X \in \twosilt (\injcat \Gamma)$.
    Then $\Sub H^0 (X) \in \ftorf_I \Gamma$ if and only if $I \in \add X$.
\end{proposition}

We define the subset $\twosilt_I(\injcat \Gamma) = \{X \in \twosilt(\injcat \Gamma) \mid I \in \add X\}$.

\begin{proposition} \label{correspondence 2-silt_I and torf_I}
    We have a commutative diagram
    \[\begin{tikzcd}[ampersand replacement=\&,cramped]
        {\twosilt_{I} (\injcat \Gamma)} \& {\twosilt (\injcat \Gamma) \setminus \twosilt_{I} (\injcat \Gamma)} \\
        {(\ftorf_I \Gamma)\opcat} \& {(\ftorf \Gamma \setminus \ftorf_I \Gamma)\opcat}
        \arrow["{\mu^+_{I}}", from=1-1, to=1-2]
        \arrow["{\Sub H^0}"', from=1-1, to=2-1]
        \arrow["{\Sub H^0}", from=1-2, to=2-2]
        \arrow["{g\opcat}"', from=2-1, to=2-2]
    \end{tikzcd}\]
    where the vertical maps are isomorphisms of posets.
    
    In particular, this implies that the map
    $\mu^+_{I} : \twosilt_{I} (\injcat \Gamma) \to \twosilt (\injcat \Gamma) \setminus \twosilt_{I} (\injcat \Gamma)$
    is order-preserving.
\end{proposition}

The following is a direct corollary using Proposition~\ref{torf has flip-flop decomposition} and Lemma~\ref{opposite flip-flop}.

\begin{corollary} \label{flip-flop for 2-silt inj}
    The poset $\twosilt(\injcat \Gamma)$ has a flip-flop decomposition:
    \[\twosilt(\injcat \Gamma) \simeq \Bigl(\twosilt_I(\injcat \Gamma) \sqcup \twosilt(\injcat \Gamma) \setminus \twosilt_I(\injcat \Gamma), \leq^{\mu_I^+}_+\Bigr).\]
\end{corollary}

\begin{proposition} \label{intervals for 2-silt inj}
    We have
    \begin{gather*}
        \twosilt_I (\injcat \Gamma) = \bigl[D\Gamma, I \oplus \Sigma\inv J\bigr] \\
        \twosilt (\injcat \Gamma) \setminus \twosilt_I (\injcat \Gamma) = \bigl[\Sigma\inv D\Gamma, J \oplus S'\bigr]
    \end{gather*}
    where $S' = (J' \stackrel{c}{\to} I) \in \Kb(\injcat \Gamma)$ is the complex constructed from
    any given minimal right $(\add J)$-approximation $c$ of $I$.
\end{proposition}

\subsection{Comparison of the decompositions under 1-APR tilt}

Let $\Lambda$ be a finite dimensional algebra admitting a simple projective module $P$,
and decompose $\Lambda = P \oplus R$ as a direct sum of $\Lambda$-modules.
We further assume that $\idim P = 1$,
so that $T = \tau\inv P \oplus R$ is a 1-APR tilting module.
We set $\Gamma = \End_\Lambda(T)$.
We know from Proposition~\ref{simple injective Gamma module} that $\Gamma$ admits a simple injective module $I = \Ext^1_\Lambda(T,P)$,
and moreover $\pdim I = 1$.
We decompose $D\Gamma = I \oplus J$ as a direct sum of $\Gamma$-modules.

The $\Lambda$-module $T$ is tilting,
so the functor $F = \Hom_\Lambda (T, \blank) : \modcat \Lambda \to \modcat \Gamma$
induces a derived equivalence $\mathbf R F = \mathbf R \Hom_\Lambda(T, \blank) : \deriv (\Lambda) \iso \deriv (\Gamma)$.
We use the following property of this derived functor:
if $M \in \modcat \Lambda$, then the object $\mathbf R F(M) \in \deriv (\Gamma)$ satisfies
\[H^i (\mathbf R F(M)) \simeq \Ext^i_\Lambda (T, M)\]
as $\Gamma$-modules, for $i \in \mathbb Z$.
In particular, we have $H^i(\mathbf R F (M)) = 0$ for $i<0$.

The functor $\mathbf R F$ restricts to an equivalence between the subcategories $\thick T \subseteq \deriv (\Lambda)$
and $\thick \Gamma = \Kb (\projcat \Gamma) \subseteq \deriv (\Gamma)$.
Now the triangle $P \to R \to \tau\inv P \to \Sigma P$ in $\deriv (\Lambda)$ shows that
$\thick T = \thick \Lambda = \Kb(\projcat \Lambda)$,
so $\mathbf R F$ restricts to an equivalence $\Kb (\projcat \Lambda) \iso \Kb (\projcat \Gamma)$.
From it we define a new functor
\[\Phi = \Sigma\inv \circ \nu_\Gamma \circ \mathbf R F : \Kb (\projcat \Lambda) \iso \Kb (\injcat \Gamma)\]
where $\nu_\Gamma : \Kb(\projcat \Gamma) \iso \Kb(\injcat \Gamma)$ is the derived Nakayama functor of $\Gamma$,
and $\Sigma$ denotes the translation functor of $\Kb(\injcat \Gamma)$
(see \ref{interpretation Sigma-1 nu} for an interpretation of the functor $\Sigma\inv \circ \nu$ in terms of torsion pairs).
Since $\Phi$ is a triangle equivalence,
it induces an isomorphism of posets $\silt \Kb(\projcat \Lambda) \simeq \silt \Kb(\injcat \Gamma)$.
We use the functor $\Phi$ to translate results about the poset $\twosilt (\injcat \Gamma) \subseteq \silt \Kb(\injcat \Gamma)$
and its flip-flop decomposition to statements about an isomorphic subposet of $\silt \Kb (\projcat \Lambda)$.

Recall the following 2-term presilting complexes:
$S = (P \xrightarrow{\nu\inv a} \nu\inv J_1) \in \Kb(\projcat \Lambda)$ where $\nu\inv a$ is a minimal left $R$-approximation of $P$,
the silting complex $\widetilde T = S \oplus R \in \Kb(\projcat \Lambda)$
(see \ref{S is projective resolution of tau-P}),
and $S' = (F J_1 \xrightarrow{\nu b} I) \in \Kb(\injcat \Gamma)$ where $\nu b$ is a minimal right $J$-approximation of $I$
(see \ref{intervals for 2-silt inj} and \ref{nu b is approx}).
Also recall our conventions on denoting 2-term complexes,
which were introduced in \ref{convention 2-term complexes}.

\begin{lemma} \label{I = Sigma P}
    The functor $\mathbf R F$ sends the stalk complex $P$ in degree 0
    to the stalk complex $I$ in degree 1.
    In other words, we have $\mathbf R F (\Sigma P) \simeq I$ in $\deriv (\Gamma)$.
\end{lemma}

\begin{proof}
    We have $\Ext^i_\Lambda(T,P) = 0$ for $i > 1$ since $\idim P = 1$,
    and we know that $\Hom_\Lambda(T,P) = 0$.
    This implies that the complex $\mathbf R F(P) \in \deriv(\Gamma)$ has cohomology concentrated in degree 1,
    and we have $H^1 (\mathbf R F(P)) \simeq \Ext^1 (T,P) = I$.
    Now the complex $\mathbf R F(\Sigma P)$ has cohomology concentrated in degree 0,
    and it is isomorphic in $\deriv(\Gamma)$ to $\Ext^1_\Lambda(T,P) = I$.
\end{proof}

\begin{lemma} \label{action of Phi on the bounds}
    We have $\Phi(\widetilde T) = \Sigma\inv D\Gamma$,
    $\Phi(\Sigma \Lambda) = J \oplus S'$,
    and $\Phi(\Sigma S \oplus R) = I \oplus \Sigma\inv J$.
\end{lemma}

\begin{proof}
    First we compute $\Phi(R)$.
    We know that $\Ext^i_\Lambda(T,R) = 0$ for $i > 0$,
    since $R$ is a summand of $T$ and $T$ is tilting.
    So the complex $\mathbf R F(R)$ has cohomology concentrated in degree 0,
    hence it is isomorphic in $\deriv(\Gamma)$ to $\Hom (T,R) = F(R)$.
    Moreover, $F(R)$ is a projective $\Gamma$-module.
    So we get $\mathbf RF(R) = (0 \to FR) \in \Kb(\projcat \Gamma)$.
    Now by Lemma~\ref{nu F(R) = J}, we know that $\nu_\Gamma (FR) = J$,
    hence $\Phi(R) = \Sigma\inv \nu (\mathbf RF (R)) = (0 \to J) \in \Kb(\injcat \Gamma)$.

    We apply the same reasoning to compute $\Phi(S)$.
    By Remark \ref{S is projective resolution of tau-P}, we know that $S$ is isomorphic in $\deriv (\Lambda)$ to $\tau\inv P$,
    which is a direct summand of $T$.
    So we get $\mathbf R F(S) = (0 \to F(\tau\inv P))$ in $\Kb(\projcat \Gamma)$.
    Again by Lemma~\ref{nu F(R) = J}, we know that $\nu_\Gamma F(\tau\inv P) = I$,
    so we get $\Phi(S) = (0 \to I)$ in $\Kb(\injcat \Gamma)$.

    We deduce that
    \[\Phi(\widetilde T) = \Phi(S \oplus R) = (0 \to I) \oplus (0 \to J) = (0 \to D\Gamma) = \Sigma\inv D\Gamma,\]
    and that
    \[\Phi(\Sigma S \oplus R) = (I \to 0) \oplus (0 \to J) = I \oplus \Sigma\inv J.\]

    We know from Lemma~\ref{I = Sigma P} that we have  $\mathbf R F (\Sigma P) \simeq I$ in $\deriv (\Gamma)$.
    From Proposition~\ref{simple injective Gamma module} we get a projective resolution of $I$,
    which implies that $\mathbf RF(\Sigma P) = (F(\nu\inv J_1) \stackrel{b}{\to} F(\tau\inv P)) \in \Kb(\projcat \Gamma)$.
    Then by Lemma~\ref{nu F(R) = J} and Lemma~\ref{nu and F commute},
    we have $\Phi(\Sigma P) = (F J_1 \xrightarrow{\nu b} I) = S' \in \Kb(\injcat \Gamma)$.
    So finally we get
    \[\Phi(\Sigma \Lambda) = \Phi(\Sigma R) \oplus \Phi(\Sigma P) = J \oplus S'.\qedhere\]
\end{proof}

In Figure \ref{fig:double flip-flop silting} we represent the situation in the poset $\silt \Kb(\projcat \Lambda)$.

\begin{figure}
    \def\svgwidth{0.4\textwidth}
    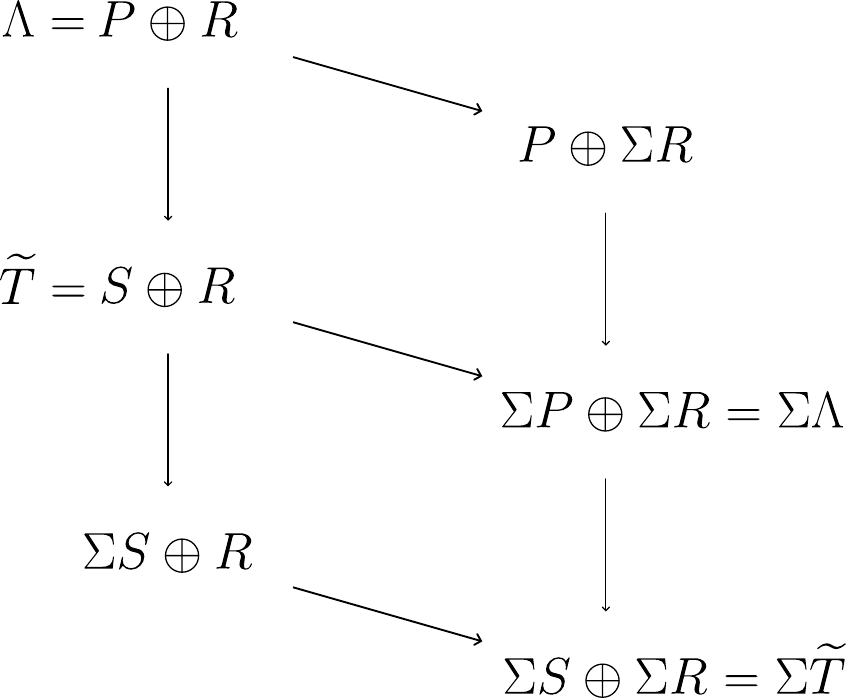
    \caption{The arrows represent the partial order,
        with the convention that an arrow $x \to y$ means $x \geq y$,
        and the vertical arrows are given by mutation, so they correspond to covering relations.
    }
    \label{fig:double flip-flop silting}
\end{figure}

\begin{corollary} \label{flip-flop for translated 2-silt inj}
    The functor $\Phi$ induces an isomorphism of posets between the interval $\bigl[\Sigma \widetilde T, \widetilde T\bigr]$ in $\silt \Kb(\projcat \Lambda)$
    and the poset $\twosilt (\injcat \Gamma)$.
    
    Moreover, this interval has a flip-flop decomposition given by
    \[\bigl[\Sigma \widetilde T, \widetilde T\bigr] \simeq \Bigl(\bigl[\Sigma \widetilde T, \Sigma S \oplus R\bigr] \sqcup \bigl[\Sigma \Lambda, \widetilde T\bigr], \leq^{\mu^+_{\Sigma S}}_+\Bigr)\]
\end{corollary}

\begin{proof}
    We know that the poset $\twosilt (\injcat \Gamma)$ is equal to the interval $\bigl[D\Gamma, \Sigma\inv D\Gamma\bigr]$
    inside the poset $\silt \Kb(\injcat \Lambda)$
    (see \ref{2-silt proj/inj as interval}),
    that it has a flip-flop decomposition (see \ref{flip-flop for 2-silt inj}),
    and that the subposets involved in the flip-flop decomposition are intervals (see \ref{intervals for 2-silt inj}).
    Namely, we have
    \begin{gather*}
        \twosilt(\injcat \Gamma) \simeq \Bigl(\twosilt_I(\injcat \Gamma) \sqcup \twosilt(\injcat \Gamma) \setminus \twosilt_I(\injcat \Gamma), \leq^{\mu_I^+}_+\Bigr) \\
        \twosilt_I (\injcat \Gamma) = \bigl[D\Gamma, I \oplus \Sigma\inv J\bigr] \\
        \twosilt (\injcat \Gamma) \setminus \twosilt_I (\injcat \Gamma) = \bigl[\Sigma\inv D\Gamma, J \oplus S'\bigr].
    \end{gather*}
    Now Lemma~\ref{action of Phi on the bounds} gives the action of $\Phi$ on the bounds of the intervals involved.
    It remains to see how the order-preserving map $\mu^+_I$ defining the flip-flop structure is translated in this context.
    First, by definition we have $\twosilt_I (\injcat \Gamma) = \{X \in \twosilt(\injcat \Gamma) \mid I \in \add X\}$.
    So every object in the interval $\bigl[\Sigma \widetilde T, \Sigma S \oplus R\bigr]$ admits $\Phi\inv(I) = \Sigma S$ as a direct summand.
    This ensures that the map $\mu_{\Sigma S}^+$ is well-defined on this interval,
    and it is then clear that $\Phi$ commutes with mutation, since it is a triangle equivalence. 
\end{proof}

\begin{remark}
    The elements of the interval $\bigl[\Sigma \widetilde T, \widetilde T\bigr]$ correspond to the silting objects in $\Kb(\projcat \Lambda)$ which are 2-term with respect to $\widetilde T$
    (see \ref{prop:def 2-term wrt}).
\end{remark}

\begin{proposition} \label{X-parts of the flip-flops are isomorphic}
    There exists an isomorphism of posets $\alpha : \bigl[P \oplus \Sigma R, \Lambda\bigr] \to \bigl[\Sigma \widetilde T, \Sigma S \oplus R\bigr]$
    satisfying $\mu_{\Sigma S}^+ \circ \alpha = \mu_P^-$,
    so that we have a commutative diagram
    \[\begin{tikzcd}
        {\bigl[P \oplus \Sigma R, \Lambda\bigr]} & {\bigl[\Sigma \Lambda, \Sigma \widetilde T\bigr]} \\
        {\bigl[\Sigma \widetilde T, \Sigma S \oplus R\bigr]} & {\bigl[\Sigma \Lambda, \Sigma \widetilde T\bigr]}
        \arrow["{{\mu^-_P}}", from=1-1, to=1-2]
        \arrow["\alpha"', from=1-1, to=2-1]
        \arrow[equals, from=1-2, to=2-2]
        \arrow["{{\mu^+_{\Sigma S}}}", from=2-1, to=2-2]
    \end{tikzcd}.\]
\end{proposition}

\begin{proof}
    We start by constructing the map $\alpha : \bigl[P \oplus \Sigma R, \Lambda\bigr] \to \bigl[\Sigma \widetilde T, \Sigma S \oplus R\bigr]$.
    Let $X$ be an object in $\bigl[P \oplus \Sigma R, \Lambda\bigr] = \twosilt_P (\projcat \Lambda)$,
    and write $X = P \oplus X'$.
    We have $\mu^-_P(X) = Z \oplus X'$ for some $Z \in \Kb(\projcat \Lambda)$. 
    We know that $\mu^-_{P}(X) \in \bigl[\Sigma \Lambda, \widetilde T\bigr]$ by Proposition~\ref{correspondence 2-silt_P and tors_P}.

    We now claim that $\mu^-_Z \circ \mu^-_{P}(X) \in \bigl[\Sigma \widetilde T, \Sigma S \oplus R\bigr]$.
    Indeed, by Proposition~\ref{prop:two complements twosilt},
    the presilting complex $X'$ admits exactly two complements in the interval $\bigl[\Sigma \widetilde T, \widetilde T\bigr]$.
    So exactly one of $\mu^+_Z(Z \oplus X')$ and $\mu^-_Z(Z \oplus X')$ belongs again to this interval.
    But we know that $\mu^+_Z(Z \oplus X') = X$ is in the interval $\bigl[P \oplus \Sigma R, \Lambda\bigr]$,
    which is disjoint from $\bigl[\Sigma \widetilde T, \widetilde T\bigr]$.
    So we get that $\mu^-_Z(Z \oplus X') \in \bigl[\Sigma \widetilde T, \widetilde T\bigr]$.
    Moreover, we have $\mu^-_Z(Z \oplus X') \notin \bigl[\Sigma \Lambda, \Lambda\bigr]$,
    since otherwise $X'$ would have three complements in $\twosilt (\projcat \Lambda)$.
    So we have $\mu^-_Z \circ \mu^-_{P}(X) \in \bigl[\Sigma \widetilde T, \Sigma S \oplus R\bigr]$.
    
    We then set $\alpha(X) = \mu^-_Z \circ \mu^-_{P}(X)$,
    so that $\alpha : \bigl[P \oplus \Sigma R, \Lambda\bigr] \to \bigl[\Sigma \widetilde T, \Sigma S \oplus R\bigr]$ is a well-defined map.
    We observe that for $X \in \bigl[P \oplus \Sigma R, \Lambda\bigr]$,
    we have $\alpha(X) = \Sigma S \oplus X'$ :
    indeed, we know that the elements of $\bigl[\Sigma \widetilde T, \Sigma S \oplus R\bigr]$ all admit $\Sigma S$ as a direct summand,
    and we have $\Sigma S \notin \add X'$ since $\Sigma S$ is not 2-term with respect to $\Lambda$.
    This implies that $\mu_{\Sigma S}^+ \circ \alpha = \mu_P^-$.

    We now show that $\alpha$ is order-preserving.
    Let $X, Y \in \bigl[P \oplus \Sigma R, \Lambda\bigr]$ such that $X \geq Y$,
    and write $X = P \oplus X'$, $Y = P \oplus Y'$.
    So we have $\alpha(X) = \Sigma S \oplus X'$, $\alpha(Y) = \Sigma S \oplus Y'$.
    It is then easy to check that we have $\Sigma S \oplus X' \geq \Sigma S \oplus Y'$
    using the definition of the partial order and the fact that $\alpha(X)$ and $\alpha(Y)$ are silting.

    Finally, one can then repeat this process dually in order to define an order-preserving map
    $\beta : \bigl[\Sigma \widetilde T, \Sigma S \oplus R\bigr] \to \bigl[P \oplus \Sigma R, \Lambda\bigr]$,
    so that for $X \in \bigl[\Sigma \widetilde T, \Sigma S \oplus R\bigr]$,
    the object $\beta (X)$ will have the form \mbox{$\beta(X) = \mu^+_Z \circ \mu^+_{\Sigma S}(X)$}
    for some $Z \in \Kb(\projcat \Lambda)$.
    It only remains to check that $\alpha$ and $\beta$ are mutual inverses,
    which is clear since left and right mutation are mutual inverses.
\end{proof}

\begin{theorem} \label{theorem ff}
    The posets $\ftors \Lambda$ and $\ftors \Gamma$ are related by a flip-flop.
\end{theorem}

\begin{proof}
    We know that $\ftors \Lambda \simeq \twosilt (\projcat \Lambda)$
    and that $\ftors \Gamma \simeq \twosilt (\injcat \Gamma)$
    by Proposition~\ref{correspondence twosilt/torsion/torsionfree}.
    Moreover, we have flip-flop decompositions
    \[\twosilt (\projcat \Lambda) \simeq \Bigl(\bigl[P \oplus \Sigma R, \Lambda\bigr] \sqcup \bigl[\Sigma \Lambda, \widetilde T\bigr], \leq_-^{\mu^-_P}\Bigr)\]
    by Corollary~\ref{flip-flop for 2-silt proj} and Proposition~\ref{intervals for 2-silt proj}
    and
    \[\twosilt (\injcat \Gamma) \simeq \Bigl(\bigl[\Sigma \widetilde T, \Sigma S \oplus R\bigr] \sqcup \bigl[\Sigma \Lambda, \widetilde T\bigr], \leq^{\mu^+_{\Sigma S}}_+\Bigr)\]
    by Corollary~\ref{flip-flop for translated 2-silt inj}.
    Finally, Proposition~\ref{X-parts of the flip-flops are isomorphic} shows that the two flip-flop data are isomorphic,
    hence the posets are related by a flip-flop.
\end{proof}

\subsection{Application to Cambrian lattices}

In this section, we explain how to apply our results to Cambrian lattices.
If $\Lambda$ is a representation-finite hereditary algebra of Dynkin type $\Delta$ with orientation $\Omega$,
the poset of torsion classes $\Cam_{\Delta,\Omega} \coloneqq \tors \Lambda$ is called a \emph{Cambrian lattice}
(since we are in the representation-finite case, all torsion classes are functorially finite).
We refer to \cite[Introduction]{Rea} for the general definition of Cambrian lattices.
For instance, for $\Lambda$ of equioriented type $A_n$,
the poset $\tors \Lambda$ is isomorphic to the well-studied Tamari lattice $\operatorname{Tam}_{n+1}$.
The following is a corollary of Theorem~\ref{theorem ff},
and generalises a result of Ladkani (Corollary~1.3 in \cite{Lad-cluster_tilting})
for simply-laced Dynkin diagrams to all (finite type) Dynkin diagrams.

\begin{corollary}
    Let $\Delta$ be a Dynkin diagram, let $\Omega, \Omega'$ be two orientations of $\Delta$.
    Then the Cambrian lattices $\Cam_{\Delta, \Omega}$ and $\Cam_{\Delta, \Omega'}$ are related by a series of flip-flops.

    Consequently, all the Cambrian lattices of the same Dynkin type are derived equivalent.
\end{corollary}

\begin{proof}
    According to \cite[Section 4.2]{Ringel},
    any two hereditary algebras $\Lambda, \Lambda'$ of type $\Delta$ can be related by a series of APR tilts
    (which are 1-APR tilts since the algebras are hereditary).
    Hence we can apply Theorem~\ref{theorem ff} iteratively to deduce that the posets $\tors \Lambda$ and $\tors \Lambda'$ are related by a series of flip-flops.
\end{proof}

We remark that one could also consider the algebras $H(\Delta, D, \Omega)$ introduced by Gei{\ss}, Leclerc and Schr{\"o}er in \cite{GLS}
(where $D$ is a symmetriser for the generalised Cartan matrix associated to the Dynkin diagram $\Delta$).
Those algebras are in general not hereditary, but they can be defined over any field.
According to \cite[Corollary~1.3]{Gyo}, the Cambrian lattice $\Cam_{\Delta, \Omega}$ is isomorphic to the lattice of torsion classes $\tors H(\Delta, D, \Omega)$.
For instance, for $\Delta = C_n$ and $D$ the corresponding minimal symmetriser,
the algebra $H(\Delta, C, \Omega)$ is given by the quiver
\[\begin{tikzcd}[cramped]
	1 & 2 & 3 & \cdots & {n-1} & n
	\arrow[from=1-1, to=1-2]
	\arrow[from=1-3, to=1-2]
	\arrow[from=1-3, to=1-4]
	\arrow[from=1-4, to=1-5]
	\arrow[from=1-6, to=1-5]
	\arrow["\varepsilon", from=1-6, to=1-6, loop, in=55, out=125, distance=10mm]
\end{tikzcd}\]
modulo the relation $\varepsilon^2 = 0$,
where the orientation of the arrows is given by $\Omega$.
It is easy to see that all orientations can be obtained from the linear one by a series of mutations at sinks or sources at the vertices $\{1, \dots, n-1\}$.
Now all these mutations correspond to 1-APR tilts, since the only relation defining the algebra involves the loop at vertex $n$.

\hidepagenumber
\section{Flip-flop for arbitrary torsion classes}
\showpagenumber

From now until the end of the paper, the goal is to show Theorem~\ref{thm intro arbitrary} from the introduction,
relating posets of arbitrary torsion classes for algebras related by a 1-APR tilt.
For that purpose, we use the framework of extriangulated categories and of first negative extensions.
All the extriangulated categories that we will be concerned with will arise as extension-closed subcategories of a given triangulated category $\mathcal D$.
Thus we will not need to use the definitions of extriangulated categories and of first negative extensions.
We explain what the general vocabulary means in this special case.

We fix a Hom-finite triangulated category $\mathcal D$ with translation functor $\Sigma$.
Let $\mathcal C$ be a subcategory of $\mathcal D$ which is closed under extensions.
A \emph{conflation} in $\mathcal C$ is a sequence $X \to Y \to Z$ in $\mathcal C$ which is a triangle in $\mathcal D$.
The \emph{first negative extension functor} $\mathbb E\inv : \mathcal C\opcat \times \mathcal C \to \modcat k$
is the restriction to $\mathcal C\opcat \times \mathcal C$ of the functor $\Hom_\mathcal D (\blank, \Sigma\inv \blank) : \mathcal D\opcat \times \mathcal D \to \modcat k$.

For two subcategories $\mathcal X, \mathcal Y \subseteq \mathcal C$
satisfying $\Hom_\mathcal C (\mathcal X, \mathcal Y) = 0$,
recall we had defined in \ref{def star notation} the subcategory $\mathcal X * \mathcal Y \subseteq \mathcal D$.
Since $\mathcal C$ is closed under extensions inside $\mathcal D$,
we still have $\mathcal X * \mathcal Y \subseteq \mathcal C$.

\subsection{Extriangulated categories and s-torsion pairs}

In this section we give necessary definitions and results about s-torsion pairs in an extriangulated category with first negative extension,
following \cite{AET}.
Along the way we also recall facts about t-structures, which we see as a special case of s-torsion pairs.

\begin{definition} \label{def s-torsion pair}
    A pair $(\mathcal T, \mathcal F)$ of subcategories of $\mathcal C$ is called an \emph{s-torsion pair} when
    \begin{enumerate}
        \item $\Hom_\mathcal C (\mathcal T, \mathcal F) = 0$
        \item $\mathcal C = \mathcal T * \mathcal F$
        \item $\mathbb E\inv (\mathcal T, \mathcal F) = 0$.
    \end{enumerate}
    
    We write $\stors \mathcal C$ for the set of s-torsion pairs in $\mathcal C$,
    and we define a partial order $\leq$ on it by setting $(\mathcal T_1, \mathcal F_1) \leq (\mathcal T_2, \mathcal F_2) \iff \mathcal T_1 \subseteq \mathcal T_2 \iff \mathcal F_1 \supseteq \mathcal F_2$,
    where the second equivalence is a consequence of the next proposition.
\end{definition}

\begin{proposition}[\protect{\cite[Proposition~3.2]{AET}}]
    Let $(\mathcal T, \mathcal F)$ be an s-torsion pair in $\mathcal C$.
    Then we have $\rightperp {\mathcal T} = \mathcal F$ and $\leftperp {\mathcal F} = \mathcal T$.
    This implies that $\mathcal T$ and $\mathcal F$ are both closed under extensions.
\end{proposition}

\begin{definition}[\protect{\cite{BBDG}}]
    A pair $(\mathcal U, \mathcal V)$ of subcategories of the triangulated category $\mathcal D$ is called a \emph{t-structure} when
    \begin{enumerate}
        \item $\Hom_\mathcal D (\mathcal U, \mathcal V) = 0$
        \item $\mathcal D = \mathcal U * \mathcal V$
        \item $\Sigma \mathcal U \subseteq \mathcal U$,
    \end{enumerate}
    where the third condition can be replaced by the equivalent condition $\Sigma\inv \mathcal V \subseteq \mathcal V$.

    The \emph{heart} of the t-structure $(\mathcal U, \mathcal V)$ is the subcategory $\mathcal H = \mathcal U \cap \Sigma \mathcal V$ of $\mathcal D$.
    It is an abelian category of $\mathcal D$,
    and we have cohomology functors $H^i : \mathcal D \to \mathcal H$ for $i \in \mathbb Z$.
\end{definition}

\begin{proposition}[\protect{\cite[Lemma~3.3]{AET}}] \label{s-tors for triangulated}
    Let $(\mathcal T, \mathcal F)$ be a pair of subcategories of $\mathcal D$,
    seen as an extriangulated category with first negative extension.
    Then $(\mathcal T, \mathcal F) \in \stors \mathcal D$ if and only if $(\mathcal T, \mathcal F)$ is a t-structure.
\end{proposition}

\begin{example}
    If $\mathcal D = \deriv (\Lambda)$ is the derived category of an algebra $\Lambda$,
    then we have the \emph{standard t-structure} $(\mathcal D^{\leq 0} (\Lambda), \mathcal D^{> 0} (\Lambda))$ given by
    \begin{align*}
        \mathcal D^{\leq 0} (\Lambda) &= \{X \in \deriv (\Lambda) \mid H^i(X) = 0, \forall i > 0\} \\
        \mathcal D^{> 0} (\Lambda) &= \{X \in \deriv (\Lambda) \mid H^i(X) = 0, \forall i \leq 0\}.
    \end{align*}
    Its heart $\mathcal H = \mathcal D^{\leq 0} (\Lambda) \cap \Sigma \mathcal D^{> 0} (\Lambda) = \mathcal D^{\leq 0} (\Lambda) \cap \mathcal D^{\geq 0} (\Lambda)$
    is the subcategory of complexes concentrated in degree 0,
    so that we have $\mathcal H \simeq \modcat \Lambda$.
\end{example}

\begin{proposition} \label{cohomology functors are representable}
    The cohomology functors $H^i_\Lambda$ for the standard t-structure on $\deriv (\Lambda)$ are representable.
    Namely, for $i \in \mathbb Z$, we have isomorphisms of functors
    \[H^i_\Lambda \simeq \Hom_\mathcal D (\Sigma^{-i} \Lambda, \blank) \simeq D \Hom_\mathcal D (\blank, \Sigma^{-i} D \Lambda)\]
    where $\Lambda$ and $D \Lambda$ denote complexes concentrated in degree 0.
\end{proposition}

\begin{remark} \label{stors for abelian}
    Let $(\mathcal U, \mathcal V)$ be a t-structure in $\mathcal D$,
    set $\mathcal H = \mathcal U \cap \Sigma \mathcal V$.
    Then $\mathcal H$ is closed under extensions inside $\mathcal D$,
    so we may see it as an extriangulated category.
    Conflations in $\mathcal H$ then correspond to short exact sequences in $\mathcal H$ in the abelian sense.
    
    Moreover, for any $X, Y \in \mathcal H$,
    we have $\mathbb E\inv (X, Y) = 0$ (since $X \in \mathcal U$ and $Y \in \mathcal V$),
    so comparing Definitions~\ref{def s-torsion pair} and \ref{def torsion pair},
    we see that s-torsion pairs in $\mathcal H$ are just ordinary torsion pairs.
\end{remark}

\begin{proposition}[\protect{\cite[Proposition~3.7]{AET}}] \label{functoriality of conflations}
    Let $(\mathcal T, \mathcal F)$ be an s-torsion pair in $\mathcal C$.
    For every $X \in \mathcal C$,
    there exists a conflation $tX \to X \to fX$ in $\mathcal C$ with $tX \in \mathcal T$ and $fX \in \mathcal F$,
    and such a conflation is unique up to isomorphism.

    Moreover, these conflations define functors $t : \mathcal C \to \mathcal T$ and $f : \mathcal C \to \mathcal F$,
    which are respectively right adjoint to the inclusion $\mathcal T \to \mathcal C$ and left adjoint to the inclusion $\mathcal F \to \mathcal C$.

    Moreover, given a morphism $h : X \to Y$ in $\mathcal C$,
    we have a commutative diagram
    \[\begin{tikzcd}
        {t X} & X & {f X} \\
        {t Y} & Y & {f Y}
        \arrow[from=1-1, to=1-2]
        \arrow["{t(h)}"', from=1-1, to=2-1]
        \arrow[from=1-2, to=1-3]
        \arrow["h", from=1-2, to=2-2]
        \arrow["{f(h)}", from=1-3, to=2-3]
        \arrow[from=2-1, to=2-2]
        \arrow[from=2-2, to=2-3]
    \end{tikzcd}\]
    in which $t(h)$ is the unique morphism making the left square commute
    and $f(h)$ is the unique morphism making the right square commute.
\end{proposition}

\begin{definition}
    Let $(\mathcal T_1, \mathcal F_1) \leq (\mathcal T_2, \mathcal F_2)$ be two comparable torsion pairs in $\mathcal{C}$,
    so that we have an interval $\mathcal{I} = \bigl[(\mathcal T_1, \mathcal F_1), (\mathcal T_2, \mathcal F_2)\bigr] \subseteq \stors \mathcal{C}$.
    We call the subcategory $\mathcal H_\mathcal I = \mathcal T_2 \cap \mathcal F_1 \subseteq \mathcal C$ the \emph{heart} of the interval $\mathcal{I}$.
\end{definition}

\begin{remark} \label{heart of interval is extension-closed}
    Since both $\mathcal T_2$ and $\mathcal F_1$ are closed under extensions,
    we see that $\mathcal{H}_\mathcal{I}$ is extension-closed in $\mathcal C$,
    hence also in $\mathcal D$.
    So we may see it as an extriangulated category with first negative extension.
\end{remark}

\begin{theorem}[\protect{\cite[Theorem~3.9]{AET}}] \label{theorem reduction extriangulated}
    Let $\mathcal{I} = \bigl[(\mathcal T_1, \mathcal F_1), (\mathcal T_2, \mathcal F_2)\bigr]$ be an interval in $\stors \mathcal C$,
    and denote by $\mathcal H_\mathcal I \subseteq C$ its heart.
    Then we have mutually inverse isomorphisms of posets
    \[\red_\mathcal{I} : \mathcal I \longrightleftarrows \stors \mathcal H_\mathcal I : \lift_\mathcal{I}\]
    given by 
    \[\red_\mathcal{I} (\mathcal T, \mathcal F) = (\mathcal T \cap \mathcal F_1, \mathcal F \cap \mathcal T_2)\]
    for $(\mathcal T_1, \mathcal F_1) \leq (\mathcal T, \mathcal F) \leq (\mathcal T_2, \mathcal F_2)$,
    and \[\lift_\mathcal{I} (\mathcal X, \mathcal Y) = (\mathcal T_1 * \mathcal X, \mathcal Y * \mathcal F_2)\]
    for $(\mathcal X, \mathcal Y) \in \stors \mathcal H_\mathcal I$.
\end{theorem}

\begin{remark}
    For $(\mathcal X, \mathcal Y) \in \stors \mathcal H_\mathcal I$,
    we have in particular $\mathcal X \in \mathcal F_1$ and $\mathcal Y \in \mathcal T_2$.
    Now the fact that $\Hom_\mathcal C (\mathcal T_1, \mathcal F_1) = 0$ and $\Hom_\mathcal C (\mathcal T_2, \mathcal F_2) = 0$
    ensures that $\mathcal T_1 * \mathcal X$ and $\mathcal Y * \mathcal F_2$
    are well-defined as subcategories of $\mathcal C$.
\end{remark}

\begin{example}
    If $\mathcal D$ is a triangulated category and $(\mathcal{U}, \mathcal{V})$ is a t-structure on $\mathcal{D}$,
    then we have an interval $\mathcal I = \bigl[(\Sigma \mathcal{U}, \Sigma \mathcal{V}), (\mathcal{U}, \mathcal{V})\bigr] \subseteq \stors \mathcal D$.
    The heart of this interval is then $\mathcal H_\mathcal I = \mathcal U \cap \Sigma \mathcal V$,
    that is, the heart of the t-structure.
\end{example}

\subsection{An extriangulated category associated to a 1-APR tilt} \label{sec:preliminaries on C}

Let $\Lambda$ and $\Gamma$ be two algebras related by a 1-APR tilt,
that is, there is a 1-APR tilting $\Lambda$-module $T$ associated to a simple projective $\Lambda$-module $P$ and $\Gamma \simeq \End_\Lambda (T)$.
We know by Corollary~\ref{1-APR tilting derived equivalence} that the derived categories $\deriv (\Lambda)$ and $\deriv (\Gamma)$ are equivalent,
via the functor $\mathbf R F = \mathbf R \Hom_\Lambda (T, \blank)$.
In particular, the functor $\mathbf R F$ sends the object $T \in \deriv(\Lambda)$,
where the $\Lambda$-module $T$ is seen as a complex concentrated in degree 0,
to the object $\Gamma \in \deriv(\Gamma)$, seen again as a complex concentrated in degree 0.
Moreover, by Lemma~\ref{I = Sigma P},
it sends the simple projective $\Lambda$-module $P$ concentrated in degree 0
to the simple injective $\Gamma$-module $I = \Ext^1_\Lambda (T,P)$ concentrated in degree 1.
In other words, we have $\Sigma \circ \mathbf R F(P) \simeq I$ in $\deriv (\Gamma)$.

To make the symmetry of the situation more apparent,
we want to work in a triangulated category $\mathcal D$ which we see as the common derived category of $\Lambda$ and $\Gamma$.
More precisely, we set $\mathcal D = \deriv (\Lambda)$,
and we consider the t-structure $(\mathcal U, \mathcal V)$ of $\mathcal D$ which is the standard t-structure of $\deriv (\Lambda)$.
We define a second t-structure $(\mathcal U', \mathcal V')$ of $\mathcal D$
by asking $(\mathbf R F (\mathcal U'), \mathbf R F(\mathcal V'))$ to be the standard t-structure of $\deriv (\Gamma)$.
The hearts $\mathcal H = \mathcal U \cap \Sigma \mathcal V$ and $\mathcal H' = \mathcal U' \cap \Sigma \mathcal V'$
of these two t\nobreakdash-structures are abelian subcategories of $\mathcal D$
which are equivalent to $\modcat \Lambda$ and $\modcat \Gamma$ respectively.

These t-structures come equipped with cohomology functors.
We denote by $H^i_\Lambda$ ($i \in \mathbb Z$) the cohomology functors associated to the t-structure $(\mathcal U, \mathcal V)$.
Since $(\mathcal U, \mathcal V)$ is the standard t-structure of $\deriv (\Lambda)$,
we know by Proposition~\ref{cohomology functors are representable} that these functors are representable,
namely we have isomorphisms of functors
\[H^i_\Lambda \simeq \Hom_\mathcal D (\Sigma^{-i} \Lambda, \blank) \simeq D \Hom_\mathcal D (\blank, \Sigma^{-i} D \Lambda).\]
Similarly, $(\mathbf R F (\mathcal U'), \mathbf R F(\mathcal V'))$ is the standard t-structure of $\deriv (\Gamma)$
and $\mathbf R F$ is a triangle equivalence.
So the cohomology functors associated to $(\mathcal U', \mathcal V')$,
which we denote by $H^i_\Gamma$,
are represented in $\mathcal D$ by shifts of the object $(\mathbf R F)\inv (\Gamma) = T \in \deriv (\Lambda)$
and co-represented by shifts of the object $W \coloneqq (\mathbf R F)\inv (D\Gamma) = \Sigma P \oplus \nu R \in \deriv (\Lambda)$
(see \ref{def of W}).
Namely we have isomorphisms of functors
\[H^i_\Gamma \simeq \Hom_\mathcal D (\Sigma^{-i} T, \blank) \simeq D \Hom_\mathcal D (\blank, \Sigma^{-i} W).\]

We define a subcategory $\mathcal C$ of $\mathcal D$ by setting $\mathcal C = \mathcal U \cap \Sigma \mathcal V'$.
We will see in Corollary~\ref{interval J} that $\mathcal C$ is the heart of a particular interval in $\stors \mathcal D$,
hence it is extriangulated with first negative extension by Remark~\ref{heart of interval is extension-closed}.
We think of $\mathcal C$ as being obtained from ``gluing'' the categories $\modcat \Lambda$ and $\modcat \Gamma$ along the functor $F : \modcat \Lambda \to \modcat \Gamma$
(see for instance \ref{C = H union H'}).

In Proposition~\ref{flip-flop intervals tors} and Proposition~\ref{flip-flop intervals torf},
we introduced distinguished torsion pairs in $\modcat \Lambda$ and $\modcat \Gamma$
as bounds of the intervals appearing in the flip-flop decompositions of $\tors \Lambda$ and $\torf \Gamma$.
Namely, we have torsion pairs $(\add P, \rightperp P)$ and $(\leftperp P, \add P)$ in $\modcat \Lambda$
and $(\leftperp I, \add I)$ and $(\add I, \rightperp I)$ in $\modcat \Gamma$.
We set $(\mathcal T, \mathcal F) = (\add P, \rightperp P)$ and $(\mathcal T', \mathcal F') = (\leftperp I, \add I)$,
where we think of $\mathcal T$ and $\mathcal F$ as subcategories of $\mathcal H \subseteq \mathcal D$
and of $\mathcal T'$ and $\mathcal F'$ as subcategories of $\mathcal H' \subseteq \mathcal D$.
The isomorphism $\Sigma \circ \mathbf R F(P) \simeq I$ in $\deriv (\Gamma)$
then becomes an equality $\Sigma \mathcal T = \mathcal F'$ of subcategories of $\mathcal D$.

\begin{remark} \label{F and T' are abelian}
    We know from Lemma~\ref{F is abelian} that $\mathcal F$ is a Serre subcategory of $\mathcal H$
    and from Lemma~\ref{T' is abelian} that $\mathcal T'$ is a Serre subcategory of $\mathcal H'$.
    In particular, $\mathcal F$ and $\mathcal T'$ are both abelian categories.
\end{remark}

\begin{example} \label{example extriangulated A3}
    We take $\Lambda = k (1 \to 2 \to 3)$ and $P = P_1$ the projective indecomposable at the vertex 1,
    so that $\Gamma \simeq k (1 \leftarrow 2 \to 3)$.

    In Figure \ref{derived strip}, we draw the Auslander--Reiten quivers of $\deriv (\Lambda)$ and $\deriv (\Gamma)$.
    We do not draw the irreducible morphisms to avoid cluttering the picture.
    The vertical alignment corresponds to the identification given by the functor $\mathbf R F$.
    The t-structures $(\mathcal U, \mathcal V$) of $\deriv (\Lambda)$ and $(\mathcal U', \mathcal V')$ of $\deriv (\Gamma)$ are indicated with solid black lines.
    Their respective hearts $\mathcal H = \mathcal U \cap \Sigma \mathcal V$ and $\mathcal H' = \mathcal U' \cap \Sigma \mathcal V'$ are indicated with solid blue lines.
    In each of these hearts, the torsion pairs $(\mathcal T, \mathcal F)$ and $(\mathcal T', \mathcal F')$ are indicated in the following way:
    the torsion part is circled with a solid red line and filled with dots,
    and the torsion-free part is circled with a dashed red line and not filled.
\end{example}

\begin{figure}
    \def\svgwidth{0.95\textwidth}
    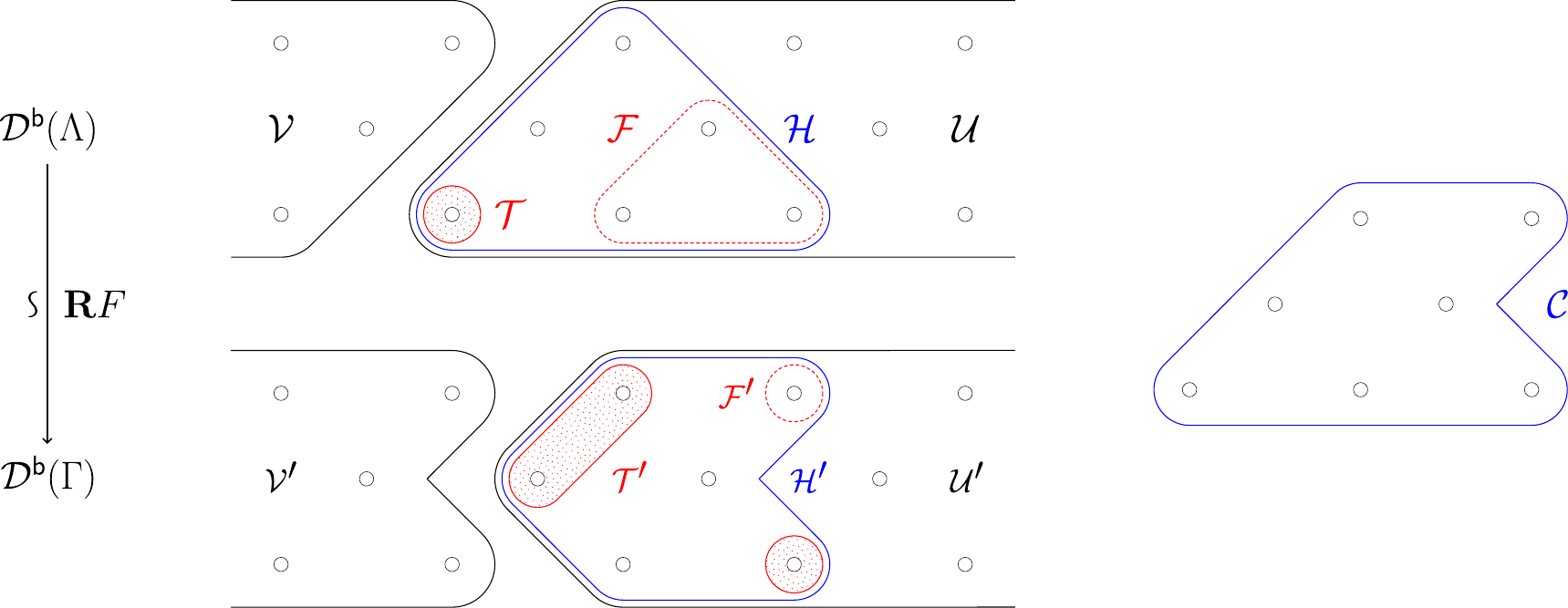
    \caption{The situation in Example \ref{example extriangulated A3}.}
    \label{derived strip}
\end{figure}

\begin{lemma} \label{def of W}
    Setting $W = \Sigma P \oplus \nu R \in \deriv (\Lambda)$,
    we have $\mathbf R F (W) = D\Gamma$.
\end{lemma}

\begin{proof}
    Decompose $\Gamma = I \oplus J$ as direct sum of injective $\Gamma$-modules.
    We know that $\mathbf R F(\Sigma P) \simeq I$ in $\deriv (\Gamma)$,
    so it suffices to compute $\mathbf R F(\nu R)$.
    Now $\nu R$ is an injective $\Lambda$-module,
    so the only non-vanishing cohomology of the complex $\mathbf R F(\nu R) \in \deriv (\Gamma)$
    is in degree 0, and it is equal to $F(\nu R) = J$ by Lemma~\ref{nu F(R) = J} and Lemma~\ref{nu and F commute}.
    So we get $\mathbf R F(\nu R) \simeq J$ in $\deriv(\Gamma)$,
    which completes the proof.
\end{proof}

\begin{lemma} \label{U' included in U}
    We have $\mathcal U' \subseteq \mathcal U$,
    or equivalently $\Sigma \mathcal V' \supseteq \Sigma \mathcal V$.
\end{lemma}

\begin{proof}
    We show that $\Sigma \mathcal V' \supseteq \Sigma \mathcal V$.
    Let $X \in \Sigma \mathcal V$.
    We want to show that $H^i_\Gamma(X) = 0$ for all $i > 0$,
    that is, $\Hom_\mathcal D (\Sigma^{-i} T, X) = 0$ for all $i > 0$.
    Take $Y \in \Kbplus (\injcat \Lambda)$ such that $X \simeq Y$ in $\deriv(\Lambda)$.
    Since $X \in \Sigma \mathcal V$,
    we can choose $Y$ so that it is concentrated in non-negative degrees.

    Then we get 
    \[\Hom_\mathcal D (\Sigma^{-i} T, X) \simeq \Hom_\mathcal D (\Sigma^{-i} T, Y) \simeq \Hom_{\Kb (\modcat \Lambda)} (\Sigma^{-i} T, Y)\]
    since $Y \in \Kb(\injcat \Lambda)$.
    Now for $i < 0$,
    we have $\Hom_{\Kb (\modcat \Lambda)} (\Sigma^{-i} T, Y) = 0$ 
    since $\Sigma^{-i} T$ is concentrated in degree $i$ and $Y$ is concentrated in non-negative degrees.
\end{proof}

\begin{corollary} \label{interval J}
    We have an interval $\mathcal J = \bigl[(\Sigma \mathcal U', \Sigma \mathcal V'), (\mathcal U, \mathcal V)\bigr]$,
    whose heart $\mathcal H_\mathcal J$ is equal to $\mathcal U \cap \Sigma \mathcal V' = \mathcal C$.
\end{corollary}

We start by proving a technical lemma from which we will derive insight about the structure of $\mathcal C$.

\begin{lemma} \label{important lemma}
    We have $(\ind \mathcal U) \setminus (\ind \mathcal U') = \{P\}$ and $(\ind \Sigma \mathcal V') \setminus (\ind \Sigma \mathcal V) = \{I\}$.
    In other words, we have $\mathcal U = \mathcal U' \plus \mathcal T$
    and $\Sigma \mathcal V' = \Sigma \mathcal V \plus \mathcal F$,
    where we recall that $\mathcal U' \plus \mathcal T$ denotes the subcategory whose set of indecomosables is $\ind \mathcal U' \cup \ind \mathcal T$.
\end{lemma}

\begin{proof}
    We only prove the first equality, since the second equality is dual to it.
    
    Let $X \in \mathcal U$ be an indecomposable object such that $X \notin \mathcal U'$.
    This implies that there exists some $j > 0$ such that $H^j_\Gamma(X) \neq 0$,
    or equivalently $\Hom_\mathcal D (X, \Sigma^{-j} W) \neq 0$,
    where $W = \Sigma P \oplus \nu R$.
    Since $\nu R$ is a direct summand of $D\Lambda$,
    we see that $\Hom_\mathcal D (X, \Sigma^{-j} \nu R)$ is a direct summand of $\Hom_\mathcal D (X, \Sigma^{-j} D\Lambda) \simeq D H^j_\Lambda (X)$.
    Moreover, we have $H^j_\Lambda (X) = 0$ since $j>0$ and $X \in \mathcal U$ by assumption.
    This implies that we must have $\Hom_\mathcal D (X, \Sigma^{-j} (\Sigma P)) \neq 0$.

    So we get that $\Hom_\mathcal D (X, \Sigma^{-(j-1)} P) \neq 0$.
    Since $X \in \mathcal U$, it is isomorphic in $\mathcal D$ to a complex of projective $\Lambda$-modules concentrated in non-positive degrees:
    \[ X \simeq \bigl(\dots \to P_{-2} \to P_{-1} \xrightarrow{d} P_0 \to 0 \to \cdots\bigr).\]
    We see that we must then have $j = 1$,
    and that a non-zero element $\alpha \in \Hom_\mathcal D (X, P)$
    corresponds to a non-zero morphism $a : P_0 \to P$ such that $ad = 0$.
    By Lemma~\ref{morphisms with target P},
    this implies that $a$ is a split epimorphism in $\modcat \Lambda$.
    Moreover, any section of $a$ can be extended to a section of $\alpha$,
    so $P$ is isomorphic to a direct summand of $X$.
    Since $X$ is indecomposable by assumption,
    we get that $X \simeq P$, which completes the proof.
\end{proof}

\begin{corollary} \label{lemma1}
    We have $\mathcal U \cap \mathcal V' = \mathcal T$ and $\Sigma \mathcal U \cap \Sigma \mathcal V' = \mathcal F'$.
\end{corollary}

\begin{proof}
    First we recall that we have $\mathcal F' = \Sigma \mathcal T$.
    Both equalities follow from the the fact that
    $\mathcal U = \mathcal U' \plus \mathcal T$ and $\mathcal V' = \mathcal V \plus \mathcal T$
    from Lemma~\ref{important lemma},
    and the fact that $\mathcal U \cap \mathcal V = \mathcal U' \cap \mathcal V' = 0$
    since $(\mathcal U, \mathcal V)$ and $(\mathcal U', \mathcal V')$ are t-structures.
\end{proof}

\begin{corollary} \label{C = H union H'}
    We have $\mathcal C = \mathcal H \plus \mathcal H'$,
    and we can decompose $\mathcal H \plus \mathcal H' = \mathcal T \plus (\mathcal H \cap \mathcal H') \plus \mathcal F'$.
\end{corollary}

\begin{lemma} \label{no non split extensions bof}
    Let $\mathcal X \subseteq \mathcal H$ be a subcategory.
    We have $\mathcal F' * \mathcal X = \mathcal F' \plus \mathcal X$.
\end{lemma}

\begin{proof}
    It suffices to show that for every $X \in \mathcal X$, $F' \in \mathcal F'$,
    we have $\Hom_\mathcal D (X, \Sigma F') = 0$.
    Now setting $T = \Sigma\inv F' \in \mathcal T$,
    we have $\Hom_\mathcal D (X, \Sigma F') = \Hom_\mathcal D (X, \Sigma^2 T) \simeq \Ext^2_\Lambda (X,T) = 0$
    since $\mathcal T = \add P$ and $\idim P = 1$.
\end{proof}

\begin{lemma}
    The following pairs of subcategories form a chain of s-torsion pairs in $\mathcal C$:
    \[(\mathcal C, \{0\}) \geq (\mathcal H', \mathcal T) \geq (\mathcal F', \mathcal H) \geq (\{0\}, \mathcal C).\]
\end{lemma}

\begin{proof}
    We know that the following are t-structures in $\mathcal D$:
    \[(\mathcal U, \mathcal V) \geq (\mathcal U', \mathcal V') \geq (\Sigma \mathcal U, \Sigma \mathcal V) \geq (\Sigma \mathcal U', \Sigma \mathcal V').\]
    Let us show that they compare in the indicated way.
    We know from Lemma~\ref{U' included in U} that we have $(\mathcal U, \mathcal V) \geq (\mathcal U', \mathcal V')$.
    Moreover, for $X \in \ind \Sigma \mathcal U$,
    by Lemma~\ref{important lemma}, we either have $X \in \ind \Sigma \mathcal U'$ or $X \simeq \Sigma P = I$.
    In both cases, we get $X \in \mathcal U'$,
    since $\mathcal U'$ is stable under $\Sigma$,
    and $I \in \mathcal H' \subseteq \mathcal U'$.
    So we have $(\mathcal U', \mathcal V') \geq (\Sigma \mathcal U, \Sigma \mathcal V)$.
    
    Since $\mathcal D$ is triangulated,
    these t-structures are s-torsion pairs in $\mathcal D$
    (see \ref{s-tors for triangulated}).

    Now we apply the map $\red_\mathcal J : \mathcal J \iso \stors \mathcal H_\mathcal J$ defined in Theorem~\ref{theorem reduction extriangulated},
    with $\mathcal J$ being the interval $\bigl[(\Sigma \mathcal U', \Sigma \mathcal V'), (\mathcal U, \mathcal V)\bigr]$ in $\stors \mathcal D$
    (see \ref{interval J}).
    We get the following chain of s-torsion pairs in the heart $\mathcal H_\mathcal J = \mathcal U \cap \Sigma \mathcal V' = \mathcal C$:
    \[(\mathcal U \cap \Sigma \mathcal V', \mathcal V \cap \mathcal U) \geq (\mathcal U' \cap \Sigma \mathcal V', \mathcal V' \cap \mathcal U) \geq (\Sigma \mathcal U \cap \Sigma \mathcal V', \Sigma \mathcal V \cap \mathcal U) \geq (\Sigma \mathcal U' \cap \Sigma \mathcal V', \Sigma \mathcal V' \cap \mathcal U)\]
    which can be rewritten as
    \[(\mathcal C, \{0\}) \geq (\mathcal H', \mathcal T) \geq (\mathcal F', \mathcal H) \geq (\{0\}, \mathcal C)\]
    by definition of $\mathcal H$, $\mathcal H'$ and $\mathcal C$, and by Corollary~\ref{lemma1}.
\end{proof}

\subsection{Structure of its poset of s-torsion pairs and flip-flop}

We explain the strategy of the proof of Theorem~\ref{theorem arbitrary}.
First, we want to translate the flip-flop decompositions of $\tors \Lambda$ and $\torf \Gamma$ in terms of the poset $\stors \mathcal C$.
This common framework will allow us to compare the two decompositions.

We give names to the following intervals in $\stors \mathcal C$:
\begin{align*}
    \mathcal I = \bigl[(\mathcal F', \mathcal H), (\mathcal C, \{0\})\bigr] \\
    \mathcal I' = \bigl[(\{0\}, \mathcal C), (\mathcal H', \mathcal T)\bigr]
\end{align*}
and we decompose
\[\mathcal I \cup \mathcal I' = (\mathcal I \setminus \mathcal I') \sqcup (\mathcal I \cap \mathcal I') \sqcup (\mathcal I' \setminus \mathcal I)\]
with \[\mathcal I \cap \mathcal I' = \bigl[(\mathcal F', \mathcal H), (\mathcal H', \mathcal T)\bigr].\]

Since the heart $\mathcal H_\mathcal I$ of the interval $\mathcal I$ is equal to $\mathcal H \simeq \modcat \Lambda$,
we have an isomorphism of posets $\red_\mathcal I : \mathcal I \iso \stors \mathcal H_\mathcal I \simeq \tors \Lambda$
(see \ref{stors for abelian}).
Similarly we have $\mathcal H_{\mathcal I'} = \mathcal H'$,
and \mbox{$\red_{\mathcal I'} : \mathcal I' \iso \stors \mathcal H_{\mathcal I'} \simeq \tors \Gamma$}
is an isomorphism of posets.

In \ref{translation flip-flop Lambda} and \ref{translation flip-flop Gamma},
we translate the flip-flop decompositions of $\tors \Lambda$ and $\torf \Gamma$
to get flip-flop decompositions for $\mathcal I$ and for $\mathcal I'$:
\[\mathcal I \simeq \bigl((\mathcal I \setminus \mathcal I') \sqcup (\mathcal I \cap \mathcal I'), \leq_-^\varphi\bigr)\]
\[\mathcal I' \simeq \bigl((\mathcal I' \setminus \mathcal I) \sqcup (\mathcal I \cap \mathcal I'), \leq_+^\psi\bigr)\]
where $\varphi$ and $\psi$ are order-preserving maps that are defined there.

The goal is then to show that the two flip-flop data are isomorphic.
In our context, this amounts to constructing an isomorphism $\alpha : \mathcal I \setminus \mathcal I' \iso \mathcal I' \setminus \mathcal I$
such that the diagram
\begin{equation}\label{commutative diagram alpha} \tag*{$(*)$}
    \begin{tikzcd}
        {\mathcal I \setminus \mathcal I'} & {\mathcal I \cap \mathcal I'} \\
        {\mathcal I' \setminus \mathcal I} & {\mathcal I \cap \mathcal I'}
        \arrow["\varphi", from=1-1, to=1-2]
        \arrow["\alpha"', from=1-1, to=2-1]
        \arrow[equals, from=1-2, to=2-2]
        \arrow["\psi", from=2-1, to=2-2]
    \end{tikzcd}
\end{equation}
commutes.
To construct $\alpha$, we once again use reduction of intervals of s-torsion pairs.
It turns out (see \ref{translation flip-flop Lambda} and \ref{translation flip-flop Gamma})
that $\mathcal I \setminus \mathcal I'$ and $\mathcal I' \setminus \mathcal I$ are both intervals of $\stors \mathcal C$,
and that their hearts are $\mathcal H_{\mathcal I \setminus \mathcal I'} = \mathcal F$ and $\mathcal H_{\mathcal I' \setminus \mathcal I} = \mathcal T'$ respectively.

In \ref{equivalence of small hearts},
we construct explicit mutually inverse equivalences between the categories $\mathcal F$ and $\mathcal T'$.
We then use them to get mutually inverse isomorphisms between $\stors \mathcal F$ and $\stors \mathcal T'$ in \ref{abelian trick},
and in \ref{X-parts of the flip-flops are isomorphic - extriangulated} we define $\alpha$
and show the commutativity of the diagram \ref{commutative diagram alpha}.

\begin{remark}
    The intervals $\mathcal I$ and $\mathcal I'$ cover the poset $\stors \mathcal C$,
    that is, we have $\stors \mathcal C = \mathcal I \cup \mathcal I'$.
    Indeed, it suffices to see that for $(\mathcal X, \mathcal Y) \in \stors \mathcal C$
    such that $(\mathcal X, \mathcal Y) \nleq (\mathcal H', \mathcal T)$,
    we have the inequality $(\mathcal X, \mathcal Y) \geq (\mathcal F', \mathcal H)$.
    Now the condition $\mathcal X \nsubseteq \mathcal H'$ is equivalent to $\mathcal T \subseteq \mathcal X$.
    This implies that $\mathcal F' \nsubseteq \mathcal Y$,
    since we have $\mathbb E\inv (\mathcal T, \mathcal F') \neq 0$ and $(\mathcal X, \mathcal Y)$ is an s-torsion pair.
    So we get that $\mathcal Y \subseteq \mathcal H$, as desired.
\end{remark}

\begin{remark}
    The fact that the categories $\mathcal F$ and $\mathcal T'$ are equivalent can be seen as a consequence of Jasso reduction.
    However, we stress that even though these categories are equivalent,
    they are in general different as subcategories of $\mathcal C$;
    this is true already in Example~\ref{example extriangulated A3}.
    Moreover, in order to show the commutativity of the diagram \ref{commutative diagram alpha},
    we need to have an explicit isomorphism between $\mathcal I \setminus \mathcal I'$ and $\mathcal I' \setminus \mathcal I$.
\end{remark}

\begin{proposition} \label{translation flip-flop Lambda}
    The isomorphism of posets $\red_\mathcal I : \mathcal I \to \stors \mathcal H_\mathcal I \simeq \tors \Lambda$
    induces a commutative diagram
    \[\begin{tikzcd}
        {\mathcal I \setminus \mathcal I'} & {\mathcal I \cap \mathcal I'} \\
        {\tors_P \Lambda} & {\tors \Lambda \setminus \tors_P \Lambda}
        \arrow["\varphi", from=1-1, to=1-2]
        \arrow["{\red_\mathcal I}"', from=1-1, to=2-1]
        \arrow["{\red_\mathcal I}", from=1-2, to=2-2]
        \arrow["f"', from=2-1, to=2-2]
    \end{tikzcd}\]
    where the map $\varphi : \mathcal I \setminus \mathcal I' \to \mathcal I \cap \mathcal I'$ is given by
    $\varphi(\mathcal X, \mathcal Y) = \bigl(\mathcal X \cap \mathcal H', \rightperp {(\mathcal X \cap \mathcal H')}\bigr)$.

    Moreover, the subset $\mathcal I \setminus \mathcal I'$ coincides with the interval $\bigl[(\mathcal T \plus \mathcal F', \mathcal F), (\mathcal C, \{0\})\bigr]$
    in $\stors \mathcal C$,
    whose heart is $\mathcal H_{\mathcal I \setminus \mathcal I'} = \mathcal F$.
\end{proposition}

\begin{proof}
    To show that the diagram is well-defined, we first need to determine the inverse images of the subsets $\tors_P \Lambda$ and $\tors \Lambda \setminus \tors_P \Lambda$
    under the isomorphism $\red_\mathcal I : \mathcal I \iso \stors \mathcal H \simeq \tors \Lambda$.
    This is the same as determining their direct images under the inverse isomorphism $\lift_\mathcal I : \tors \Lambda \iso \mathcal I$.
    Since both subsets are intervals of $\tors \Lambda$ by Proposition~\ref{flip-flop intervals tors},
    it suffices to compute the images of their upper and lower bounds under the map $\lift_\mathcal I$.
    We have
    \[\tors_P \Lambda = \bigl[(\add P, \rightperp P), (\modcat \Lambda, \{0\})\bigr] = \bigl[(\mathcal T, \mathcal F), (\mathcal H, \{0\})\bigr]\]
    and
    \[\tors \Lambda \setminus \tors_P \Lambda = \bigl[(\{0\}, \modcat \Lambda), (\leftperp P, \add P)\bigr] = \bigl[(\{0\}, \mathcal H), (\mathcal H \cap \mathcal H', \mathcal T)\bigr]\]
    where we have adapted the notation of Proposition~\ref{flip-flop intervals tors} to the one introduced at the beginning of section~\ref{sec:preliminaries on C}.

    Since $\mathcal I = \bigl[(\mathcal F', \mathcal H), (\mathcal C, \{0\})\bigr]$,
    the map $\lift_\mathcal I$ is given by
    \[\lift_\mathcal I (\mathcal X, \mathcal Y) = (\mathcal F' * \mathcal X, \mathcal Y * \{0\}) = (\mathcal F' \plus \mathcal X, \mathcal Y)\]
    for $(\mathcal X, \mathcal Y) \in \tors \Lambda$,
    where the second equality follows from Lemma~\ref{no non split extensions bof}.
    So we get 
    \begin{align*}
        &\lift_\mathcal I (\mathcal T, \mathcal F) = (\mathcal T \plus \mathcal F', \mathcal F) \\
        &\lift_\mathcal I (\mathcal H, \{0\}) = (\mathcal C, \{0\}) \\
        &\lift_\mathcal I (\{0\}, \mathcal H) = (\mathcal F', \mathcal H) \\
        &\lift_\mathcal I (\mathcal H \cap \mathcal H', \mathcal T) = (\mathcal H', \mathcal T).
    \end{align*}
    
    Since we know that $\mathcal I \cap \mathcal I' = \bigl[(\mathcal F', \mathcal H), (\mathcal H', \mathcal T)\bigr]$,
    we see that $\red_\mathcal I (\mathcal I \cap \mathcal I') = \tors \Lambda \setminus \tors_P \Lambda$,
    which shows that the right vertical arrow is well-defined.
    Now this implies that $\red_\mathcal I (\mathcal I \setminus \mathcal I') = \tors_P \Lambda$,
    since $\mathcal I \setminus \mathcal I'$ is the complement of $\mathcal I \cap \mathcal I'$ inside $\mathcal I$.
    So the left vertical arrow is well-defined,
    and we see that $\mathcal I \setminus \mathcal I'$ coincides with the interval $\bigl[(\mathcal T \plus \mathcal F', \mathcal F), (\mathcal C, \{0\})\bigr]$ in $\stors \mathcal C$.
    Finally, the heart of this interval is by definition
    $\mathcal H_{\mathcal I \setminus \mathcal I'} = \mathcal F \cap \mathcal C = \mathcal F$.

    We now show that the diagram is commutative.
    Let $(\mathcal X, \mathcal Y) \in \mathcal I \setminus \mathcal I'$.
    By the above characterisation of $\mathcal I \setminus \mathcal I'$ as the interval $\bigl[(\mathcal T \plus \mathcal F', \mathcal F), (\mathcal C, \{0\})\bigr]$,
    we have $\mathcal T \plus \mathcal F' \subseteq \mathcal X$,
    so we can write $\mathcal X = \widetilde{\mathcal X} \plus \mathcal T \plus \mathcal F'$ for some subcategory $\widetilde{\mathcal X} \subseteq \mathcal H \cap \mathcal H'$.
    To show that ${\red}_\mathcal I \circ \varphi (\mathcal X, \mathcal Y) = f \circ \red_\mathcal I (\mathcal X, \mathcal Y)$,
    it suffices to check that the torsion part of ${\lift}_\mathcal I \circ f \circ \red_\mathcal I (\mathcal X, \mathcal Y)$
    is equal to that of $\varphi (\mathcal X, \mathcal Y)$,
    namely to $\mathcal X \cap \mathcal H'$.
    Now by definition of the involved maps,
    the torsion part of ${\lift}_\mathcal I \circ f \circ \red_\mathcal I (\mathcal X, \mathcal Y)$
    is equal to
    \begin{align*}
        ((\mathcal X \cap \mathcal H) \cap \mathcal H') \plus \mathcal F' &= ((\widetilde{\mathcal X} \plus \mathcal T) \cap \mathcal H') \plus \mathcal F' \\
        &= \widetilde{\mathcal X} \plus \mathcal F' \\
        &= \mathcal X \cap \mathcal H'
    \end{align*}
    which completes the proof.
\end{proof}

Dually, we get the following result.

\begin{proposition} \label{translation flip-flop Gamma}
    The isomorphism of posets 
    \[\red_\mathcal {I'} : \mathcal I' \to \stors \mathcal H_\mathcal {I'} \simeq \tors \Gamma \simeq (\torf \Gamma)\opcat\]
    induces a commutative diagram
    \[\begin{tikzcd}
        {\mathcal I' \setminus \mathcal I} & {\mathcal I \cap \mathcal I'} \\
        {(\torf_I \Gamma)\opcat} & {(\torf \Gamma \setminus \torf_I \Gamma)\opcat}
        \arrow["\psi", from=1-1, to=1-2]
        \arrow["{\red_\mathcal {I'}}"', from=1-1, to=2-1]
        \arrow["{\red_\mathcal {I'}}", from=1-2, to=2-2]
        \arrow["g\opcat"', from=2-1, to=2-2]
    \end{tikzcd}\]
    where the map $\psi : \mathcal I' \setminus \mathcal I \to \mathcal I \cap \mathcal I'$ is given by
    $\psi(\mathcal X', \mathcal Y') = \bigl(\leftperp{(\mathcal Y' \cap \mathcal H)}, \mathcal Y' \cap \mathcal H\bigr)$.

    Moreover, the subset $\mathcal I' \setminus \mathcal I$ coincides with the interval $\bigl[(\{0\}, \mathcal C), (\mathcal T', \mathcal T \plus \mathcal F')\bigr]$
    in $\stors \mathcal C$,
    whose heart is $\mathcal H_{\mathcal I' \setminus \mathcal I} = \mathcal T'$.
\end{proposition}

In Figure~\ref{figure stors C A3}, we give a graphical representation of the structure of the poset $\stors \mathcal C$,
which should be compared with Figure~\ref{fig:double flip-flop silting}.
On the right-hand side, we also represent the different s\nobreakdash-torsion pairs in the situation of Example~\ref{example extriangulated A3},
where the torsion part is circled with a solid line and the torsion-free part is circled with a dashed line.

\begin{figure}
    \def\svgwidth{0.9\textwidth}
    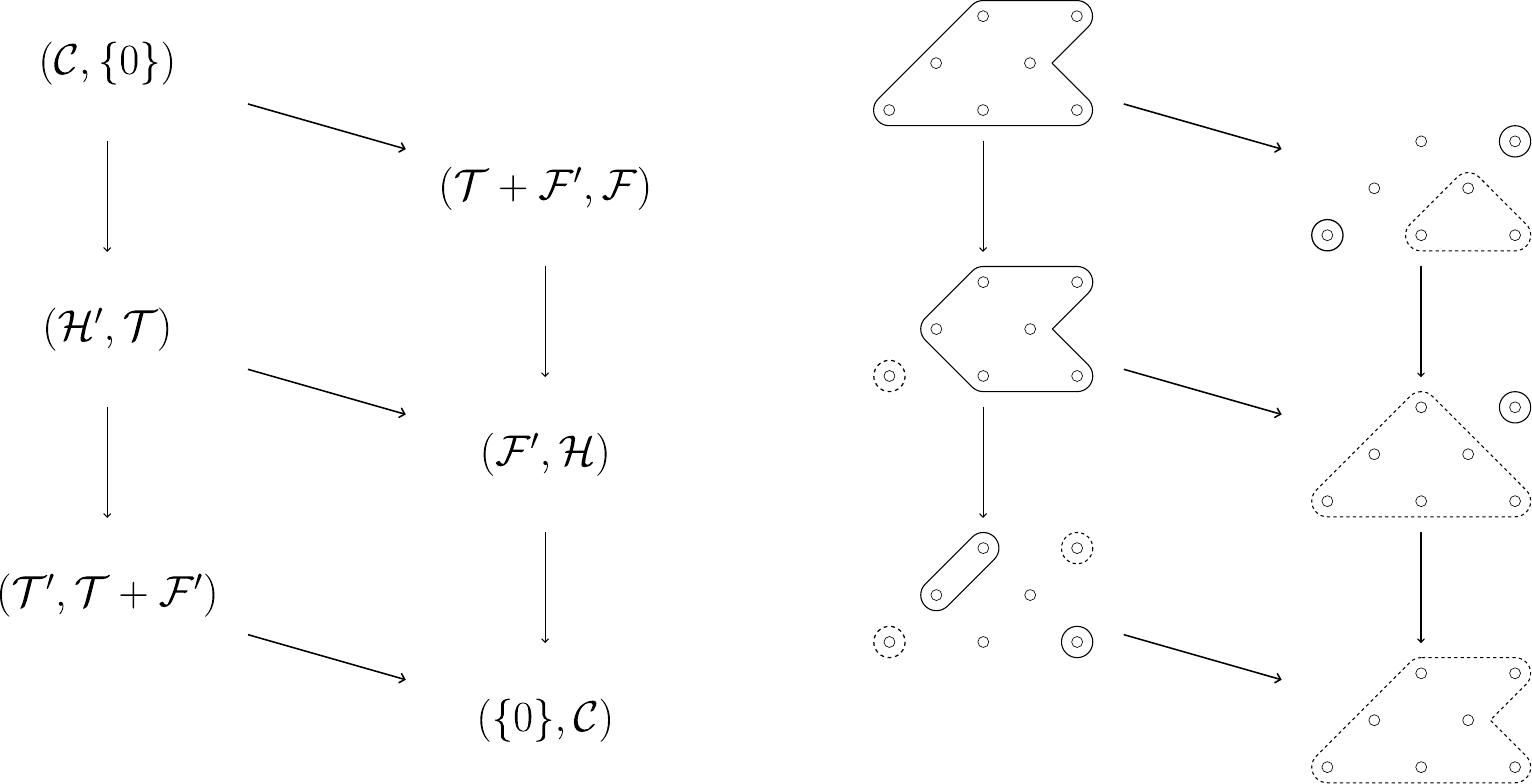    
    \caption{Structure of the poset $\stors \mathcal C$.}
    \label{figure stors C A3}
\end{figure}

\begin{proposition} \label{equivalence of small hearts}
    The functors $\theta : \mathcal C \to \mathcal T'$ and $\xi : \mathcal C \to \mathcal F$,
    defined respectively as the right adjoint of the inclusion $\mathcal T' \to \mathcal C$
    and as the left adjoint of the inclusion $\mathcal F \to \mathcal C$,
    restrict to mutually inverse equivalences of categories
    $\theta : \mathcal F \longrightleftarrows \mathcal T' : \xi$.

    In particular, for $X \in \mathcal F$ and $X' \in \mathcal T'$,
    the existence of an isomorphism $X' \simeq \theta (X)$ is equivalent to the existence of a triangle
    \[T \to X' \to X \to F'\]
    in $\mathcal D$ with $T \in \mathcal T$ and $F' = \Sigma T \in \mathcal F'$.
\end{proposition}

\begin{proof}
    The fact that $\theta$ and $\xi$ are well-defined follows from Proposition~\ref{functoriality of conflations}
    and from the fact that $(\mathcal T', \mathcal T \plus \mathcal F')$ and $(\mathcal T \plus \mathcal F', \mathcal F)$ are s-torsion pairs in $\mathcal C$.
    
    Moreover, for $X \in \mathcal F$, by definition of $\theta$,
    the object $\theta (X) \in \mathcal T'$ is characterised by the existence of a triangle of the form
    $\Sigma\inv F' \to \theta(X) \to X \to F'$ with $F' \in \mathcal F'$ (since $\Hom (X, \mathcal T) = 0$),
    and we have a dual characterisation for $\xi$.
    This implies that both compositions $\theta \xi$ and $\xi \theta$ act as the identity on objects.
    
    Similarly,
    for morphisms $h : X \to Y$ in $\mathcal F$ and $h' : X' \to Y'$ in $\mathcal T'$
    such that $X' = \theta (X)$ and $Y' = \theta (Y)$,
    we know (see \ref{functoriality of conflations}) that the equalities $h' = \theta (h)$ and $h = \xi (h')$ are both equivalent to the commutativity of the diagram
    \[\begin{tikzcd}
        {X'} & X \\
        {Y'} & Y
        \arrow[from=1-1, to=1-2]
        \arrow["{h'}"', from=1-1, to=2-1]
        \arrow["h", from=1-2, to=2-2]
        \arrow[from=2-1, to=2-2]
    \end{tikzcd}\]
    where the horizontal morphisms are the canonical ones coming from the fact that $X' = \theta(X)$ and $Y' = \theta(Y)$.
    This implies that $\xi \theta (h) = h$ and $\theta \xi (h') = h'$,
    so finally we have shown that $\theta$ and $\xi$ are mutually inverse equivalences.
\end{proof}

\begin{proposition} \label{abelian trick}
    The equivalences $\theta : \mathcal F \longrightleftarrows \mathcal T' : \xi$
    induce mutually inverse isomorphisms of posets $\stors \mathcal F \longrightleftarrows \stors \mathcal T'$,
    which we still denote by $\theta$ and $\xi$.
\end{proposition}

\begin{proof}
    We know that $\mathcal F$ and $\mathcal T'$ are both Serre subcategories,
    of $\mathcal H$ and of $\mathcal H'$ respectively (see \ref{F and T' are abelian}).
    Similarly to Remark~\ref{stors for abelian},
    this implies that s-torsion pairs in $\mathcal F$ and in $\mathcal T'$ are simply torsion pairs in the abelian sense.
    So we have $\stors \mathcal F = \tors \mathcal F$ and $\stors \mathcal T' = \tors \mathcal T'$.

    Now the equivalences $\theta$ and $\xi$ are exact (because they are equivalences between abelian categories),
    so they induce mutually inverse isomorphisms of posets $\theta : \tors \mathcal F \longrightleftarrows \tors \mathcal T' : \xi$,
    as desired.
\end{proof}

\begin{lemma} \label{clever lemma}
    Let $\mathcal Z$ be a subcategory of $\mathcal F$.
    Then we have $(\mathcal T * \mathcal Z) \cap \mathcal H' = (\theta (\mathcal Z) * \mathcal F') \cap \mathcal H$.

    Equivalently, if $\mathcal Z'$ is a subcategory of $\mathcal T'$,
    then we have $(\mathcal Z' * \mathcal F') \cap \mathcal H = (\mathcal T * \xi (\mathcal Z')) \cap \mathcal H'$.
\end{lemma}

\begin{proof}
    We start by showing that $(\mathcal T * \mathcal Z) \cap \mathcal H' \subseteq (\theta (\mathcal Z) * \mathcal F') \cap \mathcal H$
    for $\mathcal Z \subseteq \mathcal F$.
    Let $M \in (\mathcal T * \mathcal Z) \cap \mathcal H'$,
    and take a triangle $T \to M \to Z \stackrel{\beta}{\to} \Sigma T$ with $T \in \mathcal T$ and $Z \in \mathcal Z$.
    By definition of $\theta (\mathcal Z)$,
    we have a triangle $\Sigma\inv F' \to \theta (Z) \stackrel{\gamma}{\to} Z \to F'$
    with $F' \in \mathcal F'$.
    Now observe that $\Hom(\theta (\mathcal Z), \mathcal F') = 0$ since $\theta (\mathcal Z) \subseteq \mathcal T'$,
    so in particular we have $\beta \gamma = 0$.
    By the octahedral axiom,
    these two triangles fit in a commutative diagram of triangles
    \begin{equation*}
        \begin{tikzcd}
            {\theta (Z)} & {\theta (Z)} \\
            M & Z & {\Sigma T} \\
            W & {F'} & {\Sigma T}
            \arrow[equals, from=1-1, to=1-2]
            \arrow[from=1-1, to=2-1]
            \arrow["\gamma", from=1-2, to=2-2]
            \arrow[from=2-1, to=2-2]
            \arrow[from=2-1, to=3-1]
            \arrow["\beta", from=2-2, to=2-3]
            \arrow[from=2-2, to=3-2]
            \arrow[equals, from=2-3, to=3-3]
            \arrow[from=3-1, to=3-2]
            \arrow["\delta", from=3-2, to=3-3]
        \end{tikzcd}
    \end{equation*}
    for some $W \in \mathcal D$.
    
    Now since we have $M, Z, \Sigma T \in \mathcal H'$,
    we see the map $\beta$ is an epimorphism in $\mathcal H'$.
    So considering the bottom right square in the diagram as a commutative square in $\mathcal H'$,
    we see that $\delta$ is also an epimorphism in $\mathcal H'$.
    This implies that $W \in \mathcal H'$
    (it is the kernel of $\delta$ in $\mathcal H'$).
    Moreover, since $\mathcal F'$ is a torsion-free class in $\mathcal H'$,
    it is closed under subobjects,
    so we get that $W \in \mathcal F'$.
    Now the leftmost vertical triangle shows that $M \in \theta (\mathcal Z) * \mathcal F'$,
    hence $(\mathcal T * \mathcal Z) \cap \mathcal H' \subseteq \theta (\mathcal Z) * \mathcal F'$.
    Since $\mathcal H$ is closed under extensions and $\mathcal T$ and $\mathcal Z$ are both subcategories of $\mathcal H$,
    we finally get that $(\mathcal T * \mathcal Z) \cap \mathcal H' \subseteq (\theta (\mathcal Z) * \mathcal F') \cap \mathcal H$.

    A dual argument shows that for $\mathcal Z' \subseteq \mathcal T'$,
    we have $(\mathcal Z' * \mathcal F') \cap \mathcal H \subseteq (\mathcal T * \xi (\mathcal Z')) \cap \mathcal H'$.
    So combining those two inclusions together with the fact that $\theta$ and $\xi$ are mutually inverse equivalences \mbox{$\mathcal F \longrightleftarrows \mathcal T'$},
    we get the desired equalities of subcategories.
\end{proof}

\begin{proposition} \label{X-parts of the flip-flops are isomorphic - extriangulated}
    The map $\alpha = {\lift}_{\mathcal I' \setminus \mathcal I} \circ \theta \circ \red_{\mathcal I \setminus\mathcal I'} : \mathcal I \setminus \mathcal I' \to \mathcal I' \setminus \mathcal I$
    is an isomorphism of posets satisfying $\psi \alpha = \varphi$,
    so that we have a commutative diagram of morphisms of posets
    \begin{equation} \tag*{$(*)$}
        \begin{tikzcd}
            {\mathcal I \setminus \mathcal I'} & {\mathcal I \cap \mathcal I'} \\
            {\mathcal I' \setminus \mathcal I} & {\mathcal I \cap \mathcal I'}
            \arrow["\varphi", from=1-1, to=1-2]
            \arrow["\alpha"', from=1-1, to=2-1]
            \arrow[equals, from=1-2, to=2-2]
            \arrow["\psi", from=2-1, to=2-2]
        \end{tikzcd}
    \end{equation}
\end{proposition}

\begin{proof}
    We know that $\alpha$ is an isomorphism of posets, since it is defined as the composition of isomorphisms of posets.
    Moreover, its inverse is given by $\alpha\inv = {\lift}_{\mathcal I \setminus \mathcal I'} \circ \xi \circ \red_{\mathcal I' \setminus\mathcal I}$.

    We can decompose the diagram $(*)$ in the following way,
    where every map except $\varphi$ and $\psi$ is an isomorphism of posets.

    \[\begin{tikzcd}
        & {\mathcal I \setminus \mathcal I'} & {\mathcal I \cap \mathcal I'} & \\
        {\stors \mathcal F} &&& {\stors (\mathcal H \cap \mathcal H')} \\
        {\stors \mathcal T'} &&& {\stors (\mathcal H \cap \mathcal H')} \\
        & {\mathcal I' \setminus \mathcal I} & {\mathcal I \cap \mathcal I'}
        \arrow["\varphi", from=1-2, to=1-3]
        \arrow["{\red_{\mathcal I \cap \mathcal I'}}", from=1-3, to=2-4]
        \arrow["{\red\inv_{\mathcal I \setminus \mathcal I'}}", from=2-1, to=1-2]
        \arrow["{\theta}", shift left, from=2-1, to=3-1]
        \arrow[equals, from=2-4, to=3-4]
        \arrow["{\xi}", shift left, from=3-1, to=2-1]
        \arrow["{\red\inv_{\mathcal I' \setminus \mathcal I}}"', from=3-1, to=4-2]
        \arrow["\psi"', from=4-2, to=4-3]
        \arrow["{\red_{\mathcal I \cap \mathcal I'}}"', from=4-3, to=3-4]
    \end{tikzcd}\]

    To show that the diagram is commutative,
    we let $(\mathcal X, \mathcal Y) \in \stors \mathcal F$ and $(\mathcal X', \mathcal Y') \in \stors \mathcal T'$
    be such that $\theta (\mathcal X, \mathcal Y) = (\mathcal X', \mathcal Y')$,
    and we show that $\varphi \circ \lift_{\mathcal I \setminus \mathcal I'} (\mathcal X, \mathcal Y) = \psi \circ \lift_{\mathcal I' \setminus \mathcal I} (\mathcal X', \mathcal Y')$.
    Since an s-torsion pair is uniquely determined by its torsion (resp.\ torsion-free) part,
    it suffices to show that the torsion part of $\varphi \circ \lift_{\mathcal I \setminus \mathcal I'} (\mathcal X, \mathcal Y)$
    and the torsion-free part of $\psi \circ \lift_{\mathcal I' \setminus \mathcal I} (\mathcal X', \mathcal Y')$ together form an s-torsion pair in $\mathcal C$.
    So after unravelling the definitions of $\varphi$, $\psi$ and the lifting maps
    (see \ref{translation flip-flop Lambda}, \ref{translation flip-flop Gamma} and \ref{theorem reduction extriangulated}),
    we want to show that $\rightperp{\bigl(((\mathcal T \plus \mathcal F') * \mathcal X) \cap \mathcal H'\bigr)} = (\mathcal Y' * (\mathcal T \plus \mathcal F')) \cap \mathcal H$.
    
    Now using Lemma~\ref{no non split extensions bof},
    the torsion part of $\varphi \circ \lift_{\mathcal I \setminus \mathcal I'} (\mathcal X, \mathcal Y)$
    can be rewritten as
    \begin{align*}
        ((\mathcal T \plus \mathcal F') * \mathcal X) \cap \mathcal H' &= ((\mathcal T * \mathcal X) \plus \mathcal F') \cap \mathcal H' \\
        &= ((\mathcal T * \mathcal X) \cap \mathcal H') \plus \mathcal F'
    \end{align*}
    and similarly the torsion-free part of $\psi \circ \lift_{\mathcal I' \setminus \mathcal I} (\mathcal X', \mathcal Y')$
    can be rewritten as $((\mathcal Y' * \mathcal F') \cap \mathcal H) \plus \mathcal T$.

    We use one last reduction to see that it suffices to show that
    \[{\red}_{\mathcal I \cap \mathcal I'} (\varphi \circ \lift_{\mathcal I \setminus \mathcal I'} (\mathcal X, \mathcal Y)) = {\red}_{\mathcal I \cap \mathcal I'} (\psi \circ \lift_{\mathcal I' \setminus \mathcal I} (\mathcal X', \mathcal Y'))\]
    as s-torsion pairs in $\mathcal H_{\mathcal I \cap \mathcal I'} = \mathcal H \cap \mathcal H'$.
    So we want to show that $\rightperp{((\mathcal T * \mathcal X) \cap \mathcal H')} = (\mathcal Y' * \mathcal F') \cap \mathcal H$
    inside the category $\mathcal H \cap \mathcal H'$.

    We first show that $\rightperp{((\mathcal T * \mathcal X) \cap \mathcal H')} \subseteq (\mathcal Y' * \mathcal F') \cap \mathcal H$.
    Let $M \in \mathcal H \cap \mathcal H'$
    such that we have $M \in \rightperp{((\mathcal T * \mathcal X) \cap \mathcal H')}$.
    It suffices to show that $M \in \mathcal Y' * \mathcal F'$.
    Since $(\mathcal T', \mathcal F')$ is a torsion pair in $\mathcal H'$,
    we have a triangle of the form $T' \to M \to F' \to \Sigma T'$
    with $T' \in \mathcal T'$ and $F' \in \mathcal F'$.
    We show that $T' \in \mathcal Y'$,
    which will imply that $M \in \mathcal Y' * \mathcal F'$.
    For $X' \in \mathcal X'$,
    we have $\Hom(X', T') \simeq \Hom(X', M)$ since $X' \in \mathcal T'$,
    Moreover, we have $\Hom(X', M) = 0$ since $X' \in \theta(\mathcal X) \subseteq (\mathcal T * \mathcal X) \cap \mathcal H'$.
    So we get $T' \in \rightperp{\mathcal X'}$,
    which implies that $T' \in \mathcal Y'$ since $T' \in \mathcal T'$ and $(\mathcal X', \mathcal Y') \in \stors \mathcal T'$,
    as desired.

    It only remains to show that $\rightperp{((\mathcal T * \mathcal X) \cap \mathcal H')} \supseteq (\mathcal Y' * \mathcal F') \cap \mathcal H$,
    or equivalently, using Lemma~\ref{clever lemma},
    that $\Hom ((\mathcal X' * \mathcal F') \cap \mathcal H, (\mathcal Y' * \mathcal F') \cap \mathcal H) = 0$.

    We have $\Hom(\mathcal X', \mathcal Y') = 0$ since $(\mathcal X', \mathcal Y') \in \stors \mathcal F$,
    and $\Hom(\mathcal X', \mathcal F') = 0$ since $\mathcal X' \subseteq \mathcal T' \subseteq \leftperp {\mathcal F'}$.
    So we get that $\Hom(\mathcal X', \mathcal Y' * \mathcal F') = 0$,
    hence in particular $\Hom(\mathcal X', (\mathcal Y' * \mathcal F') \cap \mathcal H) = 0$.
    Moreover, we have $\Hom(\mathcal F', \mathcal H) = 0$ since $(\mathcal F', \mathcal H) \in \stors \mathcal C$,
    hence in particular $\Hom(\mathcal F', (\mathcal Y' * \mathcal F') \cap \mathcal H) = 0$.
    This implies that $\Hom(\mathcal X' * \mathcal F', (\mathcal Y' * \mathcal F') \cap \mathcal H) = 0$,
    from which the result follows.
\end{proof}

\begin{theorem} \label{theorem arbitrary}
    The posets $\tors \Lambda$ and $\tors \Gamma$ are related by a flip-flop.
\end{theorem}

\begin{proof}
    We know that $\tors \Lambda \simeq \stors \mathcal H \simeq \mathcal I$
    and that $\tors \Gamma \simeq \stors \mathcal H' \simeq \mathcal I'$.
    Moreover, we have flip-flop decompositions
    \[\mathcal I \simeq \bigl((\mathcal I \setminus \mathcal I') \sqcup (\mathcal I \cap \mathcal I'), \leq_-^\varphi\bigr)\]
    by Proposition~\ref{tors has flip-flop decomposition} and Proposition~\ref{translation flip-flop Lambda}
    and
    \[\mathcal I' \simeq \bigl((\mathcal I' \setminus \mathcal I) \sqcup (\mathcal I \cap \mathcal I'), \leq_+^\psi\bigr)\]
    by Proposition~\ref{torf has flip-flop decomposition}, Proposition~\ref{translation flip-flop Gamma} and Lemma~\ref{opposite flip-flop}.
    Finally, Proposition~\ref{X-parts of the flip-flops are isomorphic - extriangulated} shows that the two flip-flop data are isomorphic,
    hence the posets are related by a flip-flop.    
\end{proof}

\section*{References}
\printbibliography[heading=none]

\end{document}